\documentclass{amsart} 
\usepackage{amssymb}
\usepackage{amsthm}
\usepackage{dsfont}
\usepackage{xcolor}
\usepackage{cleveref}
\usepackage[shortlabels]{enumitem}
\usepackage[curve,matrix,arrow]{xy}
\theoremstyle{plain}\newtheorem{Theorem}{Theorem}[section]
\theoremstyle{plain}
\theoremstyle{plain}\newtheorem{Corollary}[Theorem]{Corollary}
\theoremstyle{plain}\newtheorem{Lemma}[Theorem]{Lemma}
\theoremstyle{plain}\newtheorem{Proposition}[Theorem]{Proposition}
\theoremstyle{definition}
\theoremstyle{definition}
\theoremstyle{definition}
\theoremstyle{definition}\newtheorem{Remark}[Theorem]{Remark}
\theoremstyle{definition}
\theoremstyle{plain}

\def\CH{{\mathcal{H}}}

\def\CN{{\mathcal{N}}}

\def\F{{\mathbb F}}
\def\Fp{{\mathbb F_p}}

\def\Z{{\mathrm Z}}
\def\B{{\mathrm B}}

\def\Aut{\mathrm{Aut}}           \def\tenk{\otimes_k}     
\def\Br{\mathrm{Br}}             \def\ten{\otimes} 

\def\dim{\mathrm{dim}}           
\def\Der{\mathrm{Der}}       
\def\End{\mathrm{End}}           

\def\Ext{\mathrm{Ext}}

 \def\gl{\mathfrak{gl}}
\def\H{\mathrm H}
\def\HH{\mathrm{HH}}
\def\Hom{\mathrm{Hom}}           
  
\def\ker{\mathrm{ker}}           
             
\def\IDer{\mathrm{IDer}}
\def\Im{\mathrm{Im}}             
\def\Ind{\mathrm{Ind}}           
\def\Inf{\mathrm{Inf}}
\def\Inn{\mathrm{Inn}}

\def\op{\mathrm{op}}
\def\Out{\mathrm{Out}}

\def\Res{\mathrm{Res}}

\def\soc{\mathrm{soc}}           
\def\Tr{\mathrm{Tr}}

\title{The Lie algebra structure of $\HH^1(B)$ for blocks with a normal 
defect group} 
\author{Markus Linckelmann, Jialin Wang} 
\date{\today}

\address{Markus Linckelmann \\
School of Science \& Technology \\
Department of Mathematics \\
City St George's, University of London \\
Northampton Square \\
London EC1V 0HB \\
United Kingdom}
\email{markus.linckelmann.1@city.ac.uk}

\address{Jialin Wang \\
School of Science \& Technology \\
Department of Mathematics \\
City St George's, University of London \\
Northampton Square \\
London EC1V 0HB \\
United Kingdom}
\email{jialin.wang@city.ac.uk}



\begin{document}

\begin{abstract}
It is an open question  whether the Lie algebra $\HH^1(B)$ of a block of a finite
group algebra $kG$ over an algebraically closed field $k$ of prime characteristic
$p$ is simple only if the block is nilpotent with an elementary abelian defect
group of order at least $3$. We show this to be the case if $B$ has a normal
abelian defect group.  Another  open question is  whether $\HH^1(kS_n)$ is  a
solvable Lie algebra, where $S_n$ denotes a symmetric group on $n$ letters. 
We give some general solvability criteria for $\HH^1(B)$ of blocks
$B$ with a normal defect group, and, making use of results of Chuang-Kessar
and Chuang-Rouquier, we deduce that if $B$ is a block of a 
symmetric group algebra with an abelian defect group, then $\HH^1(B)$ is solvable.
\end{abstract}

\maketitle
\pagestyle{myheadings}
\markright{The Lie algebra $\HH^1(B)$ for normal  defect}


\section{Introduction}

Let  $k$ an algebraically closed field of prime characteristic $p$.
Let $P$ be a finite $p$-group. By \cite[Theorem 1]{Jacobson} and 
\cite[Proposition 2.5]{LiRusimple} the Lie algebra $\HH^1(kP)$ is simple if and 
only if $P$ is elementary abelian of order at least $3$.  This led in \cite{LiRusimple}
to the question whether for a block $B$ of a finite group algebra $kG$ the
Lie algebra $\HH^1(B)$ is simple if and only if $B$ is Morita equivalent to 
$kP$ for an elementary abelian $p$-group $P$ of order at least $3$, or in
block theory parlance, if and only if $B$ is nilpotent with an elementary 
abelian defect group of order at least $3$. This is shown to hold for blocks with
a unique isomorphism classes of simple modules in \cite[Theorem 1.1]{LiRusimple}.
We show  that this holds also for blocks with a normal defect group, without the
need to specify the number $\ell(B)$ of isomorphism classes of simple $B$-modules. 
One of the motivations for stating this result is indeed the fact that this is a 
scenario in which the Lie algebra structure of $\HH^1(B)$, which is invariant under
stable equivalences of Morita type,  determines the  number $\ell(B)$. 
The required  terminology is reviewed in the Sections
\ref{quoted-Section} and \ref{block-theory-Section} below.  

\begin{Theorem} \label{thm1}
Let $G$ be a finite group and let $B$ be a block of $kG$ with a defect 
group $P$ which is normal in $G$. Then $\HH^1(B)$ is a simple Lie algebra
if and only if $B$ is a nilpotent block and $P$ is  elementary abelian 
of order at least $3$. In that case we have $\HH^1(B)\cong$ $\HH^1(kP)$.
\end{Theorem}

Brou\'e's Abelian Defect Conjecture predicts that a block $B$ with an abelian
defect group  is derived equivalent to its Brauer correspondent.
If true, then Theorem \ref{thm1} would hold for all blocks with abelian defect groups.
A derived equivalence induces a stable equivalence of Morita type. Since
the Lie algebra structure of $\HH^1(B)$ is invariant under stable equivalences
of Morita type, we can remove the normality assumption on $P$ whenever
there is a stable equivalence of Morita type between $B$ and its Brauer
correspondent.

\begin{Corollary} \label{cor1}
Let $G$ be a finite group and $B$ a block of $kG$ with a defect group
$P$ such that there is a stable equivalence of Morita type between $B$ and its
Brauer correspondent.  Then $\HH^1(B)$ is a simple Lie algebra
if and only if $B$ is a nilpotent block and $P$ is  elementary abelian 
of order at least $3$. In that case we have $\HH^1(B)\cong$ $\HH^1(kP)$.
\end{Corollary}

By a result of Puig (see e.g. \cite[Theorem 10.5.1]{LiBookII}), this applies when 
$P$ is abelian and the inertial quotient acts freely on $P\setminus \{1\}$.

\begin{Corollary} \label{cor2}
Let $G$ be a finite group and $B$ be a block of $kG$ with an abelian defect group $P$
and an inertial quotient $E$ acting freely on $P\setminus \{1\}$. Then $\HH^1(B)$ is
a simple Lie algebra if and only if $P$ is elementary abelian of order at least $3$ 
and $E$ is trivial.  In that case we have $\HH^1(B)\cong$ $\HH^1(kP)$.
\end{Corollary}

The remaining  results focus on solvability criteria for $\HH^1(B)$.  The first of these
shows that the solvability of $\HH^1(B)$ for a block with a normal defect group $P$
can be read off the solvability of certain fixed point subalgebras in $\HH^1(kP)$.

\begin{Theorem} \label{thm2}
Let $G$ be a finite group and let $B$ be a block of $kG$ with a defect 
group $P$ which is normal in $G$, and let $E$ be an inertial quotient of $B$.
Then $\HH^1(B)$ is a solvable Lie algebra if and only if $\HH^1(kP)^E$ is a
solvable Lie algebra.
\end{Theorem}

We use Theorem \ref{thm2} to prove the following three theorems,
starting with  a generalisation of \cite[Theorem 1.4]{BKL2021}.

\begin{Theorem} \label{thm3}
Let $G$ be a finite group and let $B$ be a block of $kG$ with an
extraspecial defect  group $P$ which is normal in $G$. Then $\HH^1(B)$
is a solvable Lie algebra.
\end{Theorem}

The next result improves on \cite[Theorem 1.3]{LW25}, where this was proved 
under the additional hypothesis that $P$ and $E$ are  abelian. 

\begin{Theorem} \label{thm4}
Let $G$ be a finite group and let $B$ be a block of $kG$ with a nontrivial defect 
group $P$ which is normal in $G$, and let $E$ be an inertial quotient of $B$.
\begin{itemize}
\item[{\rm (i)}] If $[P,E]=P$ and $P/\Phi(P)$ is multiplicity free as an $\Fp E$-module,
then $\HH^1(B)$ is a solvable Lie algebra. 
\item[{\rm (ii)}] Suppose that $P$ is abelian and $p$ is odd. Then $\HH^1(B)$ is a solvable Lie algebra, if and only if $[P,E]=P$ and $P/\Phi(P)$ is multiplicity free as an $\Fp E$-module.\end{itemize}
\end{Theorem}

If $[P,E]=P$ and $P$ and $E$ are abelian, then the second statement holds 
by \cite[Theorem 1.3]{LW25}. However the `only if' part of the second statement 
is not true for general $P$, not even when $E=\{1\}$. Indeed, if $p$ is odd and $P$ 
is extraspecial  of order $p^3$, then $\HH^1(kP)$ is solvable 
(cf. \cite[Theorem 1.4]{BKL2021}),  but $P/\Phi(P)$ has dimension $2$ as 
$\Fp$-vector space, and $\HH^1(kP/\Phi(P))$  is a simple Jacobson-Witt  Lie 
algebra (cf. \cite[Theorem 1]{Jacobson}).

\medskip
It is an open question whether the Lie algebra $\HH^1(kS_n)$ is solvable for any
finite symmetric group $S_n$. 
This is equivalent to asking whether the Lie algebra $\HH^1(B)$ of any  block $B$
of a finite symmetric group is solvable. 
Combining   Theorem \ref{thm4} and  results from \cite{LW25} 
with fundamental properties of blocks of symmetric groups from \cite{ChKe} and
\cite{ChRou}  we show that $\HH^1(B)$ is solvable for any block $B$ of a symmetric group
algebra with  abelian defect groups.

\begin{Theorem}\label{thm5}
Let $n$ be a positive integer and $S_n$ the symmetric group on $n$ letters.
Let $B$ be a block of $kS_n$. Suppose that $B$ has abelian defect groups.
Then the Lie algebra $\HH^1(B)$ is solvable.
\end{Theorem}

For defect groups of order $p^2$ this has been shown
previously in \cite[Proposition 5.1]{LiRusolvable}. 

\medskip
The hypothesis on $k$ being algebraically closed for some of the results stated
above  is required for K\"ulshammer's structure theorem for blocks $B$ with a 
normal defect group (see Theorem \ref{Kuenormal} below). This hypothesis
can be replaced by the weaker hypothesis of requiring $k$ to be a splitting
field for the block $B$ and its Brauer pairs.
Many intermediate results do not require $k$ to be algebraically closed, and some
do not even require $k$ to have positive characteristic, so 
we repeat the necessary  hypotheses on $k$ at the beginning of each Section.
We start with  basic facts required for this paper 
on Hochschild cohomology (Section
\ref{quoted-Section}),  block theory (Section \ref{block-theory-Section}),
finite $p$-groups and their group algebras (Section \ref{kP-Section}),
group cohomology (Section \ref{HG-Section}), and  degree $1$ Hochschild
cohomology of finite group algebras (Section \ref{HH1-Section}).
Section \ref{coprime-Section} contains results needed for one of the
key steps, in Section \ref{kPe-Section},  towards proving the main results in the
Sections \ref{main-Section} and \ref{thm1-Section}.  In an appendix we
present alternative proofs of some structural  results related to the solvability 
of $\HH^1(kP)$ for $P$ extraspecial.

\section{Notation and quoted results} \label{quoted-Section}

Let $k$ be a field. 
Given a finite dimensional associative unital $k$-algebra $A$, an $A$-$A$-bimodule 
$M$ is understood to be an $A\tenk A^\op$-module, via $amb=$ $(a\ten b)m$, for
$a$, $b\in A$ and $m\in M$.
The Hochschild cohomology of $A$ with coefficients in an $A$-$A$-bimodule
$M$ is the  graded vector space $\HH^*(A; M)=$ $\Ext_{A\tenk A^\op}(A; M)$.
In particular, $\HH^1(A)=$ $\HH^1(A;A)$ is a graded $k$-algebra. By results
of Gerstenhaber, this algebra  is  graded-commutative and admits a graded
Lie algebra structure of degree $-1$; in particular, $\HH^1(A)$ is a Lie algebra. 
There is a canonical description of $\HH^1(A;M)$  in terms
of the quotient of the space of derivations $\Der(A; M)$  by the subspace
$\IDer(A; M)$ of inner derivations which can be found in many standard sources.
In particular, the Lie algebra structure on $\HH^1(A)$ admits an explicit description
as the quotient of the Lie subalgebra $\Der(A)=$ $\Der(A;A)$ of $\End_k(A)$ by 
the Lie ideal $\IDer(A)=$ $\IDer(A;A)$ of inner derivations.  See e.g. 
\cite{Witherspoon} for an expository account on Hochschild cohomology.

\medskip
We will use without further comment the fact that $\HH^1(A)$, as a 
Lie algebra, is invariant under Morita equivalences and derived equivalences.
One of the standard tools calculating the Lie algebra structure of $\HH^1(A)$
is to consider the filtration on $\HH^1(A)$ induced by the filtration
on $A$ given by powers of the Jacobson radical $J(A)$. For $m$ a positive 
integer, we denote by $\Der_m(A)$ the Lie subalgebra of the Lie algebra of 
$k$-linear derivations $\Der(A)$  on $A$ which have image contained in $J(A)^m$.
We have the following result deduced from K\"unneth formula.

\begin{Lemma}[{\cite[Lemma 3.7]{LW25}}] \label{KuennethRad2}
Let $A$, $B$ be finite-dimensional associative unital $k$-algebras.
Suppose that the canonical map $\Der_1(B)\to$ $\HH^1(B)$ is surjective.
Then the space
$$\ Z(A) \tenk \HH^1(B) \ \oplus\ \HH^1(A)\tenk J(Z(B)),$$ 
identified to its image in $\HH^1(A\tenk B)$, 
is a Lie ideal in  $\HH^1(A\tenk B)$. In particular, if $\HH^1(A)$ is non-zero, then
$\HH^1(A\tenk B)$ is not  a simple Lie algebra. 
\end{Lemma}

For $A$ an associative  unital $k$-algebra, we denote by $[A,A]$
the $k$-span of the set of additive commutators $ab-ba$, where $a$, $b\in A$.
 A trivial verification shows that $[A,A]A=A[A,A]$; in particular, $[A,A]A$ is
 a two-sided ideal in $A$. 

\begin{Lemma}[{cf. \cite[Lemmas 2.5, 2.6]{LW25}}] \label{Amodcommutators}
Let $A$ be a finite-dimensional associative unital $k$-algebra.
\begin{itemize}
\item[{\rm (i)}]
The ideal $[A,A]A$ is the smallest ideal $I$ in $A$ such
that $A/I$ is a commutative algebra. 

\item[{\rm (ii)}]
Any derivation $f : A\to A$ preserves the space 
$[A,A]$ as well as the ideal $[A,A]A$, and $f$ induces a derivation 
$\bar f$ on $A/[A,A]A$. Moreover, if $f$ is inner,  then
$\bar f$ is zero.

\item[{\rm (iii)}]
The corrspondence $f\mapsto \bar f$ in (ii)  induces a Lie algebra homomorphism
$$\HH^1(A) \to \HH^1(\bar A) = \Der(\bar A),$$
where $\bar A=$ $A/[A,A]A$. 

\item[{\rm (iv)}]
Let $G$ be a finite group. Denote by $G'$ the commutator subgroup of $G$.
Then $[kG,kG]kG=$ $I(kG')kG$, and we have a canonical algebra isomorphism
$$kG/[kG,kG]kG \cong kG/G'.$$
\end{itemize}
\end{Lemma}

\begin{Lemma}  \label{JZAinJA2}
Let $A$ be a finite-dimensional associative unital  $k$-algebra such that 
$J(Z(A))\subseteq$  $J(A)^2$. Then every derivation on $A$ preserves $J(Z(A))$.
\end{Lemma}

\begin{proof}
Let $f : A\to A$ be a derivation. By \cite[Lemma 3.1]{LiRusimple}, $f$ preserves
$Z(A)$. Clearly $J(Z(A))=$ $Z(A)\cap J(A)$. By the assumptions we have  in fact that
$J(Z(A))=$ $Z(A)\cap J(A)^2$.
Using the Leibniz rule implies that $f(J(A)^2)\subseteq$ $J(A)$. Thus we have
$f(J(Z(A)))=$ $f(Z(A)\cap J(A)^2)\subseteq$
$Z(A)\cap J(A)=$ $J(Z(A))$, whence the result.
\end{proof}

\begin{Proposition}[{\cite[Proposition 2.3]{LW25}}]  \label{rad-Prop}
Let $A$ be a  finite-dimensional associative unital $k$-algebra. 
For $m\geq 1$, denote by $\Der_m(A)$ the subspace of $\Der(A)$ consisting
of all derivations $f : A\to$ $A$ such that 
such that $\Im(f) \subseteq$ $J(A)^m$. 
The space $\Der_2(A)$ is a nilpotent ideal in $\Der_1(A)$.
\end{Proposition}

There are several slightly different versions of this Proposition, such as 
\cite[Proposition 2.8]{LiRusolvable} or  \cite[Proposition 3.5]{LiRusimple}.
The following well-known Lemma collects basic facts on 
algebra automorphisms and the induced Lie algebra automorphisms on $\HH^1(A)$.
The elementary proof is left to the reader.

\begin{Lemma} \label{OutHHlemma}
Let $A$ be an associative unital $k$-algebra. 
Let $d : A\to A$ be a derivation and $\alpha\in$ $\Aut(A)$.
\begin{enumerate}
\item[{\rm (i)}]
The map  $\alpha\circ d \circ \alpha^{-1}$ is a derivation.

\item[{\rm (ii)}] 
If $d$ is an inner derivation,  then so is $\alpha\circ d\circ\alpha^{-1}$. More
precisely, if
$d=$ $[c,-]$ for some fixed element $c\in A$, then
 $\alpha \circ d \circ \alpha^{-1}=$ $[\alpha(c), -]$.
 
 \item[{\rm (iii)}] 
 The map $d \mapsto \alpha\circ d \circ \alpha^{-1}$ induces Lie algebra automorphisms
 on $\Der(A)$ and $\HH^1(A)$.
 
 \item[{\rm (iv)}]
 Suppose that $\alpha$ is an inner automorphism given by $\alpha(a)=cac^{-1}$,
 for some fixed element $c\in A^\times$ and all $a\in A$. Then
 $\alpha\circ d \circ \alpha^{-1} = d + [cd(c^{-1}), - ]\ .$
 
 \item[{\rm (v)}]
 The subgroup $\Inn(A)$ of inner automorphisms of $A$ acts as identity on 
 $\HH^1(A)$, and the action of $\Aut(A)$ induces an
 action of $\Out(A)=$ $\Aut(A)/\Inn(A)$ on $\HH^1(A)$ by Lie algebra automorphisms.
 \end{enumerate}
\end{Lemma}

\section{Block theory background}
\label{block-theory-Section}

Let $k$ be a field of prime characteristic $p$.
Given a finite group $G$, a block of $kG$ is an indecomposable algebra factor 
$B$ of $kG$. Note that then $B$ is an indecomposable direct summand of $kG$
as a $kG$-$kG$-bimodule.
The defect groups of $B$ are the maximal $p$-subgroups $P$ of
$G$ such that $kP$ is isomorphic to a direct summand of $B$ as a 
$kP$-$kP$-bimodule. The defect groups of $B$ form a $G$-conjugacy class of 
$p$-subgroups of $G$. 
For a fixed defect group $P$ of $B$ there is a unique block $C$ of $kN_G(P)$,
called the Brauer correspondent of $B$, such that $C$ is isomorphic to a
direct summand of $B$ as $kN_G(P)$-$kN_G(P)$-bimodules. The block $C$
has $P$ as a defect group.
If the block $B$ of $kG$
is nilpotent - a concept introduced by Brou\'e and Puig \cite{BrPunil} - then $B$ is
Morita equivalent to any of its defect group algebras $kP$, by a result of Puig
in \cite{Punil}.
The converse of this statement is not known in general, but it is known when
the defect groups are elementary abelian. That is, if $B$ is Morita equivalent to
$kP$ for some elementary abelian $p$-group, then $B$ is nilpotent, and $P$ is
isomorphic to the defect groups of $B$. See e.g. \cite[Section 10.2]{LiBookII}
for more precise statements and references. 

\medskip
The following result, due to K\"ulshammer, implies that the Hochschild cohomology
of blocks with a normal defect group can be calculated using certain twisted
group algebras.

\begin{Theorem}[{\cite{Kuenormal}}] \label{Kuenormal}
Let $G$ be a finite group and $B$ a block of $kG$ with a defect group $P$ which
is normal in $G$.  Suppose that $k$ is a splitting field for $B$ and its Brauer pairs.
Then $B$ is Morita equivalent to a twisted group algebra of the form
$k_\alpha(P\rtimes E)$ for some $p'$-subgroup $E$ of $\Aut(P)$ and some
$2$-cocycle $\alpha\in$ $\Z^2(E;k^\times)$
inflated to $P\rtimes E$ via the canonical surjection $P\rtimes E\to E$.
\end{Theorem}

It is easy to see that conversely  every algebra of the form $k_\alpha(P\rtimes E)$ 
with the notation as above is indeed a block algebra of some finite central 
$p'$-extension of $P\rtimes E$.
As a consequence of Theorem \ref{Kuenormal}, 
we will need to calculate $\HH^1(k_\alpha(P\rtimes E))$.
 This Lie algebra has a Lie subalgebra canonically isomorphic to 
$\HH^1(kP)^E$ (cf. \cite[Lemma 5.1]{LW25}). 
We denote by $[P,E]$ the subgroup of $P$ generated by the
elements of the form $u(^e{u})^{-1}$, with $u\in P$ and $e\in E$.
We recall the following facts from \cite{LW25} regarding this Lie subalgebra
in the case where $P$ is abelian. 

\begin{Lemma}\cite[Lemma 5.9]{LW25} \label{lemglk}
Let $P$ be a non-trivial finite abelian $p$-group, and set $J=J(kP)$. Let $E$ be 
a finite $p'$-group acting on $P$ such that $[P,E]=P$. Then
$\Der(kP)^E=$ $\Der_1(kP)^E$. Write
$$J/J^2 \cong \oplus_{j=1}^r\ S_j^{\oplus n_j}$$
with pairwise non-isomorphic simple $kE$-modules $S_j$ and positive 
integers  $n_j$ ($1\leq j \leq r$). Then $\F_j = \End_{kE}(S_j)$ is a 
finite-dimensional commutative  extension field of $k$, where $1\leq j\leq r$, 
and the  canonical map $\Der_1(kP)\to \End_k(J/J^2)$ induces a surjective 
Lie algebra homomorphism
$$\Der(kP)^E \to \End_{kE}(J/J^2)\cong \prod_{j=1}^r \gl_{n_j}(\F_j).$$
The kernel of this Lie algebra homomorphism is the nilpotent
ideal $(\Der_2(kP))^E$ of $E$-stable derivations on $kP$  with image in $J^2$.
\end{Lemma}

As mentioned earlier, any derivation on a finite-dimensional
$k$-algebra $A$ preserves the centre (cf. \cite[Lemma 3.1]{LiRusimple}), thereby
inducing a Lie algebra homomorphism $\HH^1(A)\to$ $\HH^1(Z(A))$. In particular,
$Z(A)$ is a module for the Lie algebra $\HH^1(A)$. For source
algebras of blocks we mention the following additional fact.

\begin{Proposition} 
Let $B$ be a block of a finite group algebra $kG$ with defect group $P$,
source idempotent $i\in$ $B^P$ and corresponding inertial quotient $E$.
Suppose that $k$ is a splitting field for $B$ and its Brauer pairs.
Set $A=iBi$. The Brauer homomorphism $\Br_P$ induces surjective
maps $A^P\to kZ(P)$ and  $Z(A)\to $ $(kZ(P))^E$. 
Every class in $\HH^1(B)$ has a representative $\tau$ which annihilates $i$,
and then $\tau|_{A}$ is a derivation on $A$. Restricting $\tau$ to $Z(A)$ induces
Lie algebra homomorphisms
$$\HH^1(B) \to \HH^1(Z(A))$$
and
$$\HH^1(B)\to \HH^1(Z(A)/(Z(A)\cap [A,A])) \cong \HH^1((kZ(P))^E).$$
\end{Proposition}

\begin{proof}
The fact that $\Br_P(A^P)\cong$ $kZ(P)$ is due to Puig (cf. \cite[Corollary 6.6.8]{LiBookII}), 
and the fact that this isomorphism sends $Z(A)$ to $(kZ(P))^E$ is
due to Brou\'e (cf.  \cite[Corollary 6.17.3]{LiBookII}).
It is well-known that any class in $\HH^1(A)$ has a representative which vanishes at
$i$ (see e.g. \cite[Lemma 4.1.(v)]{LiRusolvable}). The rest follows from the
fact that any derivation on $A$ preserves $Z(A)$ (by \cite[Lemma 3.1]{LiRusimple})
and the subspace $[A,A]$, together with the
fact that $Z(A)\cap [A,A]=$ $Z(A)\cap \ker(\Br_P)$, from \cite[Lemma 3.4]{KeLiFrob}.
\end{proof}

The argument of this proof does not show that the map $Z(A)\to$ $Z(A)/(Z(A)\cap [A,A])$
induces a Lie algebra homomorphism $\HH^1(Z(A))\to$ $\HH^1((kZ(P))^E)$.

\section{Some basic calculations in finite $p$-group algebras} 
\label{kP-Section}

For $P$ a finite $p$-group, we denote by $\Phi(P)$ the Frattini subgroup of $P$;
this is the smallest normal subgroup of $P$ such that $P/\Phi(P)$ is elementary
abelian, with the convention $\Phi(1)=1$. The subgroup $\Phi(P)$ is 
generated by commutators and $p$-th powers of elements in $P$.
We will need the following  elementary 
group theoretic facts.

\begin{Lemma} \label{p-group-Lemma}
Let $P$ be a finite $p$-group and $T$ a normal subgroup of $P$.
Let $S$ be the subgroup of $P$ containing $T$ such that
$S/T=$ $Z(P/T)$. Set $R=T\cap \Phi(P)$ and $Q=S\cap \Phi(P)$.
The following hold.
\begin{itemize}
\item[{\rm (i)}]
 We have $S/R=$ $Z(P/R)$.
\item[{\rm (ii)}]
The group $Q/R$ is non-trivial if and only if $\Phi(P)$ is
not contained in $T$, or equivalently, if and only if $R$ is a
proper subgroup of $\Phi(P)$.
\end{itemize} 
\end{Lemma}

\begin{proof}
The canonical map $P/T\to P/R$ sends $Z(P/T)$ to $Z(P/R)$, and hence
$Z(P/T)\leq$ $S/T$. Conversely, since $S/T=$ $Z(P/T)$ we have
$[P,S]\leq T$. We also have $[P,Q]\leq P'$ $\leq \Phi(P)$, and hence
$[P.S]\leq R$, whence 
$S/R\leq$ $Z(P/R)$. This proves (i). 

In order to prove (ii), assume first that $\Phi(P)\leq T$. Then $P/T$ is
abelian, so $S=P$, hence $R=$ $Q=$ $\Phi(P)$,  implying that $Q/R$ is trivial. 
Assume next that $\Phi(P)$ is not contained in $T$. In particular, $T$ is a proper
subgroup of $P$, and hence a proper subgroup of $S$. Thus $R$ is a proper
subgroup of $S$. We need to show that $Q/R$ is non-trivial.
Recall that by  (i) we have $Z(P/R)=$ $S/R$.
Moreover, $\Phi(P)/R$ is a normal subgroup of $P/R$, and it is non-trivial because 
$\Phi(P)$ is not contained in $T$, hence not contained in $R$. 
It follows that $Q/R=$ $S/R \cap \Phi(P)/R=$ $Z(P/R)\cap \Phi(P)/R$
is non-trivial. This proves (ii).
\end{proof}

\begin{Lemma}\label{L:Pe-705}
Let $P$ be a non-trivial finite $p$-group and $C$ a subgroup of $P$. 
There exists a unique maximal characteristic subgroup $T$ of
$P$ such that $T\leq C$.
\end{Lemma}

\begin{proof}
If $S$, $S'$ are characteristic subgroups of $P$  contained
in $C$, then $SS'$ is a characteristic subgroup of $P$ contained
in $C$. The Lemma follows.
\end{proof}

\begin{Lemma}\label{order3} 
Let $Q$ be a finite $p$-group and $\varphi$ an automorphism of $Q$.
If $[Q, \langle \varphi\rangle]\neq 1$, then $|Q|\geq 3$.
\end{Lemma}

\begin{proof}
This is an immediate consequence of the fact that a finite group of order $2$
has no non-trivial automorphism. 
\end{proof}

\begin{Lemma}\label{L-Jsquare}
Let $k$ be a field of prime characteristic $p$.
Let $P$ be a finite $p$-group.  Let $Q$ be a subgroup
of $\Phi(P)$. Then $J(kQ)\subseteq J(kP)^2$. In particular, we have
$J(kP')\subseteq$ $J(kP)^2$, where $P'$ is the commutator subgroup
of $P$. 
\end{Lemma}

\begin{proof} 
We may assume $Q=\Phi(P)$. For $x$, $y\in P$ we have $xy-1=$ $(x-1)(y-1)-(x-1)-(y-1)$.
Thus if $x-1$, $y-1$ are both are contained in $J(kP)^2$,  then so is $xy-1$.
Thus it suffices to show $x-1\in$ $J(kP)^2$ for $x$ running over a generating set
of $Q$. Now $Q$ is generated by elements of the form $x^p$ and commutators
$xyx^{-1}y^{-1}$, with $x$, $y\in P$. We have $xyx^{-1}y^{-1}-1=$
$(xy-yx)x^{-1}y^{-1}$, and this belongs to $J(kP)^2$ since $xy-yx$ does (see e.g.
\cite[Proposition 4.4.9]{LiBookI}). 
We have $x^p-1=$ $(x-1)^p$, which also belongs to $J(kP)^2$. 
Since $P'\leq \Phi(P)$, the last statement follows.
\end{proof}

\begin{Lemma} \label{PPhiP-Lemma}
Let $k$ be a field of prime characteristic $p$.
Let $P$ be a finite $p$-group, with commutator subgroup $P'$ and Frattini
subgroup $\Phi(P)$. 

\begin{itemize}
\item[{\rm (i)}]
Every derivation on $kP$ induces a derivation on
$kP/P'$ and on $kP/\Phi(P)$. 

\item[{\rm (ii)}]
The maps in (i)  induce Lie algebra  homomorphisms
$$\HH^1(kP)\to \HH^1(kP/P') \to   \HH^1(kP/\Phi(P)) .$$

\item[{\rm (iii)}]
Any class in the kernel of the canonical map $\HH^1(kP)\to$ $\HH^1(kP/\Phi(P))$
is represented by a derivation with image contained in $J(kP)^2$. In particular,
this kernel is a nilpotent Lie ideal. 
\end{itemize}
\end{Lemma}

\begin{proof}
Every derivation on $kP$ induces a derivation on $kP/P'$ by 
Lemma \ref{Amodcommutators} and on $kP/\Phi(P)$ by
\cite[Proposition 2.5]{LiRusimple} and its proof.  Since
$[kP,kP]\subseteq$ $J(kP)^2$, any inner
derivation on $kP$ has image contained in $J(kP)^2$. Thus the 
last statement follows from  Lemma \ref{L-Jsquare}.
\end{proof}

\begin{Proposition} \label{HomPkE}
Let $k$ be a field of prime characteristic $p$.
Let $P$ be a finite $p$-group and $E$ a $p'$-subgroup of $\Aut(P)$.
Consider $\Hom(P,k)$ as an $E$-set with $e\in E$ acting on $\zeta : P\to k$
by ${^e{\zeta}}=\zeta\circ e^{-1}$. We have $\Hom(P,k)^E\neq 0$ if and
only if $[P,E]<P$.
\end{Proposition}

\begin{proof}
Given $\zeta\in$ $\Hom(P,k)$, we have $\zeta\in$ $\Hom(P,k)^E$ if and
only if $\zeta(a^e)=\zeta(a)$ for all $u\in P$ and all $e\in E$, hence if and only if
$[P,E]\leq$ $\ker(\zeta)$. Thus if $[P,E]=P$, then $\Hom(P,k)^E=0$.
Conversely, if $[P,E]<P$, then $\Phi(P)[P,E]$ is a proper normal $E$-stable subgroup
of $P$.
 By the above, any  group homomorphism $P\to k$ with kernel containing 
 $\Phi(P)[P,E]$ is  stable under the action of $E$, whence the result.
\end{proof}

\section{Quoted results in group cohomology}
\label{HG-Section}

Let $k$ be a field of prime characteristic $p$. 
Taking cohomology of a group $G$ with coefficients in a $kG$-module $M$ is 
functorial in $M$, albeit not exact in general. 
We will use without further comment that short exact sequences of modules
over a group algebra induce long exact sequences in cohomology; see e.g.
\cite[Section 3.3]{Evens} or \cite[Ch. III, (6.1) Proposition]{Brown}.
Aside from long exact sequences induced by short exact sequences of modules,  
we will need the following low degree exact sequence from the 
Lyndon-Hochschild-Serre spectral sequence.

\begin{Theorem}[{\cite[Corollary 7.2.3]{Evens}, \cite[6.3.8]{Weibel}}]
\label{T:esH1}
Let $G$ be a group and let  $M$ be a $kG$-module and $N$ be a normal 
subgroup of $G$.  Then there is an exact sequence of abelian groups
\[0\to \H^1(G/N;M^N)\to \H^1(G;M)\to \H^1(N;M)^{G/N}\to \H^2(G/N;M^N)\to \H^2(G;M).
\] 
Moreover, the map from $\H^1(G/N;M^N)$ to $\H^1(G;M)$ is the inflation map on 
$1$-cocycles, $\inf^G_N(f)(g) = f(gN)$ for all $g\in G$ and $f \in\mathrm{Z}^1(G/N; M^N)$ 
and the map from $ \H^1(G;M)$ to $\H^1(N;M)^{G/N}$ is given by restriction to $N$.
\end{Theorem}

\begin{Corollary}
\label{L:Pe-55}
With the notation of Theorem \ref{T:esH1}, if $M^N$  is a projective $kG/N$-module, 
then restriction from $G$ to $N$ induces an isomorphism
$$\H^1(G;M))\cong \H^1(N; M)^{G/N} . $$ 
\end{Corollary}

\begin{proof}
Since $M^N$ is projective as a $kG/N$-module it  follows that $H^n(G/N; M^N)=0$ for
any positive integer $n$. The exactness of the sequence in Theorem \ref{T:esH1}
yields the isomorphism as stated.
\end{proof}

We will need the following well-known description of transfer in group cohomology
(see e.g. \cite[Sections 2.8, 3.3]{BensonI} for more details on the Eckmann-Shapiro Lemma
as well as \cite[Section 3.6]{BensonI} for transfer maps in cohomology). 

\begin{Lemma} \label{tr1-Lemma}
Let $G$ be a finite group, $H$ a subgroup of $G$.
The composition
$$\H^*(H; k) \cong \H^*(G; \Ind^G_H(k)) \to \H^*(G;k)$$
is equal to the transfer map $\Tr^G_H$, where the first 
isomorphism is from the Eckmann-Shapiro Lemma, and the 
second map is induced by the augmentation map $\Ind^G_H(k)\to k$.
\end{Lemma}

We will  use the standard descriptions of low degree group cohomology,
which can be found in many sources such as  \cite[Section 2.3]{Evens}.
In particular, a  $1$-cocycle of $G$ with
coefficients in $M$ is a map $\tau : G\to M$ satisfying $\tau(xy)=$
$\tau(x) + x\cdot \tau(y)$, for all $x$, $y\in G$. The set of $1$-cocycles
of $G$ with coefficients in $M$ is denoted by $\Z^1(G; M)$. This is a
$k$-vector space with structure induced by the $k$-vector space structure on $M$.
A $1$-coboundary of $G$ with coefficients in $M$ is a map $\tau : G\to M$
for which there exists an element $m\in M$ satisfying $\tau(x) = x\cdot m -m$
for all $x\in G$. The set of $1$-coboundaries of $G$ with coefficients in $M$
is denoted by $\B^1(G; M)$. This is a $k$-subspace of $\Z^1(G; M)$, and we have
a canonical identification $\H^1(G;M)\cong$ $\Z^1(G; M)/\B^1(G;M)$. 
A surjective homomorphism of $kG$-modules $M\to M'$ clearly induces a
surjective map $\B^1(G; M)\to$ $\B^1(G; M')$ map, but the induced 
maps $\Z^1(G; M)\to$ $\Z^1(G; M')$ and $\H^1(G;M)\to$ $\H^1(G; M')$
need not be surjective (the latter fits into a long exact sequence). 
We will frequently use the following basic comptational facts on $1$-cocycles.

\begin{Lemma} \label{L:cocycle-2}
Let $G$ be a finite group and $M$ a $kG$-module; write ${^xm}$ for the
action of $x\in G$ on $m\in M$. Let $\tau\in$ $\Z^1(G; M)$.
\begin{itemize}
\item[{\rm (i)}]
We have $\tau(1) = 0$.
\item[{\rm (ii)}]
For any $x\in G$ we have ${^x\tau(x^{-1})} = - \tau(x)$.
\item[{\rm (iii)}]
For any $x$, $y\in G$ we have
$\tau(xyx^{-1}y^{-1}) = \tau(x) + {^x\tau(y)} + {^{xy}\tau(x^{-1})} + 
{^{xyx^{-1}}\tau(y^{-1})}.$
\item[{\rm (iv)}]
For any $x$, $y\in G$, if $\tau(x)=0$ then 
$\tau(xyx^{-1}y^{-1}) = {^x\tau(y)} + {^{xyx^{-1}}\tau(y^{-1})}. $
\item[{\rm (v)}]
For any $x\in G$ and any positive integer $m$ we have 
$\tau(x^m)=$ $\sum_{j=0}^{m-1}\ x^j\cdot \tau(x)$.
\end{itemize}
\end{Lemma}

\begin{proof}
For $x$, $y\in G$, the $1$-cocycle identity is $\tau(xy)=$ $\tau(x)+{^x\tau(y)}$.
Applied with $x=1$ this yields (i). Applied with $x^{-1}$ and $x$ this yields
in turn (ii). Using the $1$-cocycle identity repeatedly  yields (iii), and (iv) is a special
case of (iii). We have
$\tau(x^m)=\tau(x\cdot x^{m-1})=$ $\tau(x) + {x\cdot \tau(x^{m-1})}$, and hence
(v) follows by induction. 
\end{proof}

The following technical statements will be used in the proof of Theorem
\ref{T:kPe-1}, one of the key results for the present paper. 

\begin{Lemma} \label{H1G-submodule}
Let $G$ be a finite group and $M$ a $kG$-module. Let $U$, $V$ be $kG$-submodules
of $M$ such that $V\subseteq$ $U\subseteq$ $M$.
Let $c\in$ $H^1(G; M)$, and denote by $d$ the image of $c$ in $\H^1(G; M/V)$ under
the map induced by the canonical surjection. 
\begin{itemize}
\item[{\rm (i)}]
Suppose that $d$ is in the image of the 
map $\H^1(G; U/V)\to $ $\H^1(G; M/V)$ induced by the
inclusion $U/V\subseteq$ $M/V$. Then $c$ is in the image of the 
map $\H^1(G; U)\to $ $\H^1(G; M)$ induced by the
inclusion $U\subseteq$ $M$.
\item[{\rm (ii)}]
If the map $\H^1(G; U/V)\to $ $\H^1(G; M/V)$ induced by the
inclusion $U/V\subseteq$ $M/V$  is surjective, then
the map $\H^1(G; U)\to $ $\H^1(G; M)$ induced by the
inclusion $U\subseteq$ $M$  is surjective.
\end{itemize}
\end{Lemma}

\begin{proof}
All maps in this proof are understood to be the maps induced
by the obvious inclusions and canonical surjections. 
Note that these induce exact sequences
$\H^1(G; U) \to$ $\H^1(G; M)\to$ $\H^1(G; M/U)$ and
$\H^1(G; U/V) \to$ $\H^1(G; M/V)\to$ $\H^1(G; M/U)$; where we
use the canonical isomorphism $(M/V)/(U/V)\cong$ $M/U$.
By the assumptions, the image of the canonical map $\H^1(G; U/V)\to$ 
$\H^1(G; M/V)$  contains the class $d$. It follows that
the canonical map $\H^1(G; M/V)\to$ $\H^1(G; M/U)$ sends $d$ to $0$. 
Thus the map $\H^1(G; M)\to$ $\H^1(G; M/U)$ sends $c$ to  zero as this map 
factors  as the composition of maps $\H^1(G;M)\to$ $\H^1(G; M/V)\to $ 
$\H^1(G; M/U)$. 
It follows that $c$ is in the image of the map $\H^1(G;U)\to$ $\H^1(G;M)$.
This proves the first statement, and the second statement is an immediate
consequence. 
\end{proof}

For $G$ a finite group, we denote by $I(kG)$ the augmentation ideal of $kG$.
If $P$ is a finite $p$-group, then $I(kP)=$ $J(kP)$.

\begin{Lemma}\label{H1-induced}
Let $G$ be a finite group, $N$ a normal subgroup, $U$ a  $kG$-module, and
$V$ a submodule of $U$ such that $I(kN)\cdot U\subseteq $ $V$. 
Let $\tau\in $ $\Z^1(G; U)$. Suppose that 
$\tau(N) \subseteq$ $V$. Then $\tau$ induces canonically an element $\zeta$
in $\Z^1(G/N; U/V)$; more precisely, there is a unique $\zeta\in$ $\Z^1(G/N; U/V)$
such that $\zeta(xN)$ is equal to the image of $\tau(x)$ in $U/V$, for
all $x\in G$. Equivalently, denoting by $\pi : U\to U/V$ the canonical
surjection, $\zeta$ is the unique element of $\Z^1(G/N; U/V)$ satisfying
$\pi\circ\tau=$ $\Inf^{G}_{G/N}(\zeta)$. Moreover, if $\tau\in$ $\B^1(G;U)$, then
$\zeta\in$ $\B^1(G/N; U/V)$.
\end{Lemma}

\begin{proof}
The hypothesis $I(kN)\cdot U\subseteq$ $V$  implies that $N$ acts as
identity on $U/V$, and hence $U/V$ inherits a $kG/N$-module structure.
We need to show that for $x\in G$ the image in $U/V$ of $\tau(x)$ depends only
on the coset $xN$. Let $y\in N$. Then $\tau(xy)=$ $\tau(x)+ x\cdot \tau(y)$.
By the assumptions, $\tau(y)$, and hence $x\cdot\tau(y)$, belong to $V$.
Thus the images of $\tau(x)$ and of $\tau(xy)$ in $U/V$ are equal. 
Thus the map $\zeta$ sending $xN$ to the image of $\tau(x)$ in $U/V$ is 
well-defined, and clearly $\zeta$ inherits from $\tau$ the $1$-cocycle indentity.
Assume now that $\tau\in$ $\B^1(G;U)$. That is, there is an element $u\in U$
such that $\tau(x) = $ $x\cdot u - u$ for all $x\in G$. Since we assume that
$\tau(N)\subseteq$ $V$ it follows that for all $y\in N$ we have 
$\tau(xy)=$  $\tau(x) + x\cdot\tau(y)=$ $x\cdot u - u + w$  for some $w\in V$
(depending on $x$ and $y$).
Therefore, if $\bar u$ is the image of $u$ in $U/V$, then $\zeta(\bar x)=$
$\bar x\cdot\bar u - \bar u$, so $\zeta\in$ $\B^1(G/N; U/V)$. The Lemma follows.
\end{proof}

\section{On the Hochschild cohomology in degree $1$ of finite group algebras}
\label{HH1-Section}

Let $k$ be a field of prime characteristic $p$. 
For $G$ a group and $M$ a $kG$-$kG$-bimodule, an application of
the Eckmann-Shapiro Lemma with respect to the diagonal image of $G$ in 
$G\times G$ yields a canonical isomorphism 
$$\HH^*(kG; M) \cong \H^*(G; M),$$
where on the right side $M$ is understood as a $kG$-module via the 
conjugation action given by ${^x{m}}=$ $xmx^{-1}$, for all $x\in G$ and $m\in M$.
We spell out without proof the following well-known facts for future reference.

\begin{Proposition} \label{HH1H1-Prop}
Let $G$ be a group, let $M$ be a $kG$-$kG$-bimodule,
and let $U$ be a $kG$-$kG$-subbimodule of $M$.
\begin{itemize}
\item[{\rm (i)}]
Let $d : kG\to M$ be a linear map, and let $\tau : G\to M$ be the map defined
by $\tau(x)=d(x)x^{-1}$ for all $x\in G$. Then $d$ is a derivation if and only if
$\tau$ is a $1$-cocycle,with respect to the conjugation action by $G$ on $M$,
and $d$ is an inner derivation if and only if $\tau$ is a $1$-coboundary. 

\item[{\rm (ii)}]
The correspondence $d\mapsto \tau$ from (i) is a  linear  isomorphisms 
$$\Der(kG; M) \cong \Z^1(G; M),$$
restricts to a linear isomorphism
$$\IDer(kG; M) \cong \B^1(G; M),$$
and induces a linear isomorphism
$$\HH^1(kG: M) \cong \H^1(G; M).$$

\item[{\rm (iii)}]
The isomorphisms in  (ii) are functorial in  $M$.
In particular, if $U$ is a $kG$-$kG$-subbimodule of $M$, then
the image of a derivation $d : kG\to M$ is  contained in $U$ if and only
if the image of the corresponding $1$-cocycle $\tau$ as in (i) is contained
in $U$.
\end{itemize}
\end{Proposition}

For $G$ a finite group and $z\in Z(G)$, the
$1$-dimensional subspace $k\cdot z$ of $kG$ is a direct summand of
$kG$ with respect to the conjugation action of $G$ on $kG$, and hence the
canonical isomorphism $k\cong k\cdot z$ induces an injective map
$\H^1(G;k) \to $ $\HH^1(kG)$. This map clearly depends on $z$. 
The direct sum taken over all elements in $Z(G)$ yields a canonical
injective map $\H^1(G; kZ(G))\to$ $\HH^1(kG)$. Since $G$ acts as identity
by conjugation on $kZ(G)$, we have a canonical isomorphism
$\H^1(G; kZ(G))\cong$ $\H^1(G;k)\tenk kZ(G)$. Thus, as a
$kZ(G)$-module, the image of the map $H^1(G; kZ(G))\to$ $\HH^1(kG)$ 
is the free $kZ(G)$-module of  rank $\dim_k(\H^1(G;k))$
with basis any $k$-basis of the image of $\H^1(G;k)$ in $\HH^1(kG)$. 

\begin{Proposition} \label{H1Gk-image-Prop}
Let $G$ be a finite group. Let $m$ be the $p$-rank of the largest
elementary abelian $p$-group quotient of $G$.
The image of the canonical injection 
$\H^1(G;k)\to$ $\HH^1(kG)$ is an abelian Lie subalgebra of dimension $m$.
The image of the canonical injection 
$\H^1(G;kZ(G))\to$ $\HH^1(kG)$ is a Lie subalgebra of dimension 
$m\cdot |Z(G)|$.
\end{Proposition}

\begin{proof}
Since $\H^1(G;k)=$ $\Hom(G, k)$ and since $k$ is an elementary abelian
$p$-group, it follows that the dimension of $\H^1(G;k)$ is equal to the 
$p$-rank of the largest elementary abelian $p$-group  quotient of $G$.
Let $\zeta$, $\eta\in$ $\H^1(G;k)$. By Proposition \ref{HH1H1-Prop}, the
images in $\HH^1(kG)$ of $\zeta$ and $\eta$ are represented by
the derivations $d$ and $e$ on $kG$ defined by $d(x)=$ $\zeta(x)x$ and 
$e(x)=\eta(x)x$, for all $x\in G$. Thus $[d,e]$ is the derivation given for
all $x\in$ $G$ by $[d,e](x)=$ $d(e(x))-e(d(x))=$ $d(\eta(x)x)-e(\zeta(x)x)=$ 
$\zeta\eta x-\eta\zeta x=0$, whence the first statement.
\end{proof}

\begin{Remark} \label{H1Gk-Remark}
The canonical image of $kZ(G)\tenk \H^1(G;k)$ in $\HH^1(kG)$ is a Lie subalgebra
which is not abelian in general.
Let $v$, $w\in$ $Z(G)$ and $\zeta$, $\eta\in$ $\Hom(G,k)$. The image of
$u\ten \zeta$ in $\HH^1(kG)$ is the class of the derivation $d$  given by
$d(g) = \zeta(g)gu$
for all $g\in G$.
Similarly, the image of $v\ten \eta$ in $\HH^1(kG)$  is the class of the derivation $e$
given by
$e(g) = \eta(g)gv$
for all $g\in G$. A short calculation shows that $[d,e]$ is given by
$$[d,e](g) = (\zeta(v)\eta(g)-\zeta(g)\eta(u)) guv$$
for all $g\in G$. If all group homomorphisms $G\to k$
vanish on $Z(G)$, then this does become an  abelian  Lie subalgebra. 
This is for instance the case when $G$ is a $p$-group and $Z(G)\leq$ $\Phi(G)$.
\end{Remark}

\begin{Remark} \label{OutRemark}
Let $G$ be a finite group and $\varphi$  a group automorphism.  Extending $\varphi$
linearly to $kG$ yields an algebra automorphism of $kG$, denoted 
by $\hat\varphi$.
As a consequence of  Lemma \ref{OutHHlemma},  we have an action of $\Out(G)$
on $\HH^1(kG)$ induced by the map sending a derivation $d$ on $kG$ to 
the derivation ${^\alpha{d}}=$ $\hat\varphi \circ d\circ \hat\varphi^{-1}$. Via the
isomorphism $\HH^1(kG)\cong$ $\H^1(G; kG)$ from Proposition \ref{HH1H1-Prop}
one checks that this action corresponds to the action induced by the map which sends 
$\tau\in$ $\Z^1(G; kG)$ to $\hat\varphi\circ\tau\circ \varphi^{-1}$. 
\end{Remark}

\section{Coprime group action and semidirect products} 
\label{coprime-Section}

Let $p$ be a prime. 
In what follows, if $E$ is a group acting on a finite $p$-group $P$,
then, for $e\in E$, we denote by $Pe$ the $P$-coset of $e$ in the semidirect
product $P\rtimes E$ corresponding to the action of $E$.  Since we calculate
in $P\rtimes E$, the notation ${^e{u}}=$ $eue^{-1}$ for the action of $e\in E$ 
on $u\in P$ makes sense. 
We will freely use
standard results on coprime group actions, such as 
\cite[Theorem 5.1.4]{Gor}, \cite[Theorem 5.3.5]{Gor}, \cite[Theorem 5.3.10]{Gor},
\cite[\S 8.2]{KurzStell}.

\begin{Lemma}\label{L:Pe-2}
Let $P$ be a finite  abelian $p$-group and let $e$ be a nontrivial 
$p'$-automorphism of $P$. Let $Q$ be a subgroup of $P$. 
 The map $Q\to P$ sending $a\in Q$ to $a{(^ea)}^{-1}$ is a group homomorphism.
 If $Q\cap C_P(e)=\{1\}$, then this group homomorphism is injective.
In particular,  if $Q=P$ and $C_P(e)=\{1\}$, then this map is an automorphism of $P$, 
and we have $a\notin \Phi(P)$ if and only if $a{(^ea)}^{-1}\notin \Phi(P)$.
\end{Lemma}

\begin{proof}
Since $P$ is abelian it is clear that the map $a\mapsto $   $a{(^ea)}^{-1}$ 
is a group homomorphism from $Q$ to $P$. Assume that $Q\cap C_P(e)=$ $\{1\}$.
If $a\in Q$ such that $a{(^ea)}^{-1}=1$, then $a={^e{a}}$, so 
$a=1$. Thus in this case this map is an injective
group homomorphism from $Q$ to $P$.
This proves the first statement, and the rest is an immediate consequence, using 
that $\Phi(P)$ is a characteristic subgroup. 
\end{proof}

\begin{Lemma} \label{L:Pe-05}
Let $P$ be a finite $p$-group, $Q$ a  subgroup of $Z(P)$, and let $e$ be
a $p'$-automorphism of $P$. For $v\in Q$ and $x\in P$ we have
${^v(xe)} = (v(^ev)^{-1})xe$.
\end{Lemma}

\begin{proof}
We have ${^v(xe)} = $ $vxev^{-1}=$ $vxev^{-1}e^{-1} e$. Now $ev^{-1}e^{-1}\in$
$Z(P)$, so this element commutes with $x$, whence the result. 
\end{proof}

\begin{Lemma}\label{L:Pe-71}
Let $P$ be a non-trivial finite $p$-group and $e$ a non-trivial $p'$-automorphism
acting trivially on $Z(P)$.  Let $T$ be the unique characteristic subgroup of $P$
which is maximal subject to satisfying $T\leq C_P(e)$ and let $S$ be the characteristic 
subgroup of $P$ containing $T$ such that $S/T=Z(P/T)$. Set $R=T\cap \Phi(P)$
and $Q=S\cap \Phi(P)$.
\begin{itemize}
\item[{\rm (i)}]
The group $R$ contains the non-trivial subgroup $Z(P)\cap \Phi(P)$.
\item[{\rm (ii)}]
We have  $S/R=Z(P/R)$, and $S$  is not contained in $C_P(e)$.
\item[{\rm (iii)}]
If $R$ does not contain  $\Phi(P)$, then $Q$ is not
contained in $C_P(e)$.
\end{itemize}
\end{Lemma}

\begin{proof}
The assumption on $e$ acting non-trivially on $P$ but trivially on
$Z(P)$ implies that $Z(P)$ is a proper subgroup of $P$, hence $P$ is non-abelian,
and therefore $\Phi(P)\neq$ $\{1\}$, implying in turn that $Z(P)\cap\Phi(P)\neq \{1\}$.
By the assumptions and Lemma \ref{L:Pe-705}  there is a unique maximal characteristic
subgroup $T$ of $P$ satisfying $Z(P) \leq$ $T\leq $ $C_P(e)$. Note that $C_P(e)$,
and hence $T$,  is a proper subgroup of $P$ as $e$ is a non-trivial automorphism of $P$.
Thus $Z(P/T)=$ $S/T$ is non-trivial, so $S$ is a characteristic subgroup of $P$ and
$T$ is a proper subgroup of $S$. It follows that  $S$ is not
contained in $C_P(e)$ by the maximality of $T$ with this property.
Since $S$, $T$ and $\Phi(P)$ are characteristic subgroups of 
$P$, so are $R$, $Q$. Since $T$ contains $Z(P)$, it follows that $R$ contains 
$Z(P)\cap\Phi(P)$.
It follows further from Lemma \ref{p-group-Lemma} that $S/R=Z(P/R)$.
Assume that $R$ does not contain $\Phi(P)$. Then, again  by Lemma 
\ref{p-group-Lemma}, $R=$ $T\cap \Phi(P)$ is a proper subgroup of $Q=$ $S \cap \Phi(P)$;
in particular, $Q$ is not contained in $R$, hence not contained in $T$
either.  Thus $T$ is a proper subgroup of $QT$, and clearly $QT$ is a characteristic
subgroup of $P$. By the maximality of 
$T$ subject to the condition $T\leq C_P(e)$ it follows that $Q$ is
not contained in $C_P(e)$. The Lemma follows.
\end{proof}

For the sake of completeness,   we include
the following Proposition which generalises \cite[Proposition 5.5]{LW25}. 

\begin{Proposition} \label{PE-prop}
Let $P$ be a  finite $p$-group and $E$ a finite  $p'$-group acting on $P$.
Suppose that for every $y\in E$ we have $[P,C_E(y)]=P$. 
Then for every $E$-stable derivation $f : kP \to k(P\rtimes E)$ we have
have $\Im(f)\subseteq J(k(P\rtimes E))$.
\end{Proposition}

\begin{proof}
Set $J=J(kP)$. Since $E$ is a $p'$-group we have $J(k(P\rtimes E))=$ 
$J\cdot k(P\rtimes E)=$ $k(P\rtimes E)\cdot J$, and 
$k(P\rtimes E)/J(k(P\rtimes E))\cong$ $kE$.
Since $f$ is a derivation, we have $f(J^2)\subseteq J\cdot k(P\rtimes E)$.

Since $J$ is $E$-stable, it follows that the restriction of $f$
to $J$ is an $E$-stable map, as is the composition with the
canonical  map $\epsilon : k(P\rtimes E) \to kE$.  Thus the
 map $\epsilon\circ f|_J  : J \to k$ sends $J^2$ to zero,
 hence factors through the canonical surjection $J\to$ $J/J^2$.
We therefore get a commutative diagram of $kE$-modules  (where $E$
acts by conjugation)
 $$\xymatrix{ J \ar[rr]^{f|_J} \ar[d] & & k(P\rtimes E)  \ar[d]^{\epsilon} \\
 J/J^2 \ar[rr]_g & & kE.}$$
 We need to show that the map $g$ is zero, which is equivalent to $f$ 
 having image in $J(k(P\rtimes E))$. 
 The conjugation action of $E$ permutes the summands of
$\oplus_{y\in E}\ kPy$, and the stabiliser in $E$ of a summand $kPy$ is clearly
equal to $C_E(y)$. Thus the right vertical column is the direct sum of maps
of the form $\Ind_{C_E(y)}^E(kPy) \to \Ind^E_{C_E(y)}(k)$. Thus $g$ is the sum
of the maps obtained from commutative diagrams of  the form
 $$\xymatrix{ J \ar[rr]^{f|_J} \ar[d] & & \Ind^E_{C_E(y)}(kPy)  \ar[d]^{\epsilon} \\
 J/J^2 \ar[rr]_{g_y} & & \Ind^E_{C_E(y)}(k) .}$$
 with $y$ running over a set of representatives of the conjugacy  classes in $E$.
 Through Frobenius reciprocity, the map $g_y$ corresponds to a
$kC_E(y)$-homomorphism $h_y : J/J^2 \to k$.
Since $[P,C_E(y)]=P$,  it follows from \cite[Lemma 2.8]{LW25}  that $J/J^2$ has 
no nonzero trivial  summand as a $kC_E(y)$-module. Since $J/J^2$ is also semi-simple 
as a $kC_E(y)$-module, this implies that $g_y$ is zero. This implies that $g=0$,
hence that $f$ sends $kP$ to $J\cdot k(P\rtimes E)=$ $J(k(P\rtimes E))$. 
\end{proof}

\section{On $\H^1(P; kPe)$} 
\label{kPe-Section}

Let $k$ be a field of prime characteristic $p$. Given a finite $p$-group $P$, a
$p'$-automorphism $e$ of $P$ and an $\langle e\rangle$-stable normal subgroup
$Q$ of $P$,  the spaces $kPe$ and $kP/Qe$ are $kP$-$kP$-bimodules.
We regard these frequently as modules for a
subgroup of $P$ via the action of $P$ by conjugation. 
The main purpose of this Section is to prove the following result, which is one
of the key steps towards proving the Theorems stated in the Introduction.

\begin{Theorem} \label{T:kPe-1}
Let $P$ be a non-trivial finite $p$-group and $e$ a non-trivial $p'$-automorphism
of $P$. Consider $kPe$ as a $kP$-module with $P$ acting by conjugation.
Thern every class in $\H^1(P; kPe)$ has a representative in $\Z^1(P; kPe)$
whose image is contained in $J(kP)^2e$.
\end{Theorem}

By Proposition  \ref{HH1H1-Prop} (iii) this is equivalent to stating that under the
hypotheses of Theorem \ref{T:kPe-1}, every class in $\HH^1(kP; kPe)$ can be represented
by a derivation $kP\to kPe$ with image contained in $J(kP)^2e$.
We break up  the proof in a series of steps.

\begin{Lemma} \label{L:Pe-1}
Let $P$ be a finite $p$-group and $Q$ a subgroup of $Z(P)$ such that
$Q\cap C_P(e)=\{1\}$.  Then $Q$ 
acts freely by conjugation on the set $Pe$, and $kPe$, with $Q$ 
acting by conjugation, is a free $kQ$-module.  In particular, we have
$$\H^1(Q; kPe) = \H^2(Q; kPe) = 0.$$
Moreover, the set $R=$ $\{vev^{-1}e^{-1}\ |\ v\in Q\}$ is a subgroup of $Z(P)$
isomorphic to $Q$, and we have 
$$(kPe)^Q  = (\sum_{w\in R} w) kPe.$$
In particular, if $Q$ has order at least $3$, then 
$(kPe)^Q \subseteq J(kP)^2e. $
\end{Lemma}

\begin{proof}
In order to show that $Q$ acts freely by conjugation on the set $Pe$, 
we  need to show that  the stabiliser of $xe$ in $Q$ is trivial. 
By Lemma \ref{L:Pe-05} the conjugation action by $v\in Q$ on $Pe$ is equal
to left multiplication with $vev^{-1}e^{-1}$. By Lemma \ref{L:Pe-2}, 
since $Q$ is a subgroup of $Z(P)$ and since $e$ stabilises $Z(P)$, it follows 
that the map sending $v\in Q$ to 
$vev^{-1}e^{-1}$  is a group homomorphism $Q \to$ $Z(P)$. 
If $v\neq 1$ then the hypothesis $Q\cap C_P(e)=$ $\{1\}$ implies
that $vev^{-1}e^{-1}\neq 1$, so that this group homomorphism is injective.
This shows that the set
$R=$ $\{vev^{-1}e^{-1}\ |\ v\in Q\}$ is a subgroup of $Z(P)$ isomorphic to $Q$, 
and hence that no non-trivial element of $Q$ stabilises 
$xe$ under the conjugation action. In other words, $Pe$ is a free $Q$-set for the
conjugation action, and thus $kPe$ is a free $kQ$-module for this action. 
In particular, we have $\H^1(Q; kPe)=0$. 
The above calculations show that if $v$ runs over all elements of $Q$, then $w=$ 
$vev^{-1}e^{-1}$ runs over the elements of the subgroup $R$. 
This shows the equality $(kPe)^Q=$ $(\sum_{w\in R}) kPe$.  
The last statement follows from the fact that $\sum_{w\in R} w$ spans $\soc(kR)$, 
which is contained in $J(kR)^2$ whenever the order of $Q$, hence the order of $R$,
 is at least $3$. 
\end{proof}

The following result proves Theorem \ref{T:kPe-1} in the case where $e$ acts
non-trivially  on $Z(P)$.

\begin{Proposition} \label{kPe-Prop-1}
Let $P$ be a finite $p$-group and $e$ a  $p'$-automorphism of $P$.
Suppose that $e$ acts non-trivially on $Z(P)$. 
Consider $kPe$ as $kP$-module with  $P$ acting by conjugation.
Then every class in $\H^1(P; kPe)$ has a representative $\zeta$ with
$\zeta(P)\subseteq$ $J(kP)^2e$.
\end{Proposition}

\begin{proof}
Set $Q=[Z(P), \langle e\rangle]$. Since $e$ acts non-trivially on $Z(P)$, the group
$Q$ is non-trivial.  By \cite[Theorem 5.2.3]{Gor} we have
 $Z(P)= Q \times C_{Z(P)}(e)$; in 
particular, we have $C_Q(e)=1$, and  $|Q|\geq 3$ (cf. Lemma \ref{order3}). 
By Lemma \ref{L:Pe-1} we have $H^1(Q; kPe)=0$. The exact sequence
$$0\to \H^1(P/Q; (kPe)^Q)\to \H^1(P;kPe)\to \H^1(Q;kPe)^{P/Q})\to \cdots$$
from Theorem \ref{T:esH1} implies that inflation from $P/Q$ to $P$ induces an
isomorphism
$$\H^1(P/Q; (kPe)^Q)\cong \H^1(P;kPe).$$
Thus every class in $\H^1(P; kPe)$ has a representative with image contained in
$(kPe)^Q$. Again by Lemma \ref{L:Pe-1}, this space is contained in $J(kP)^2e$,
whence the result.
\end{proof}

The following Lemmas will be needed for the proof of Theorem \ref{T:kPe-1}
in general.

\begin{Lemma}\label{L:Pe-5}
Let $P$ be a finite $p$-group and let 
 $e$ be a $p'$-automorphism of $P$. Consider $kPe$ as a $kP$-module with
$P$ acting by conjugation.  Let $S$ be a subgroup of $P$.
Let $R$ be an $\langle e\rangle$-stable normal
subgroup of $P$, and let $Q$ be the subgroup of $P$ containing $R$ such
that $Q/R=Z(P/R)$.
Suppose that every class in $\H^1(S; kP/Re)$ has a representative $\sigma$ in 
$\Z^1(S; kP/Re)$ satisfying $\sigma(S)\subseteq J(kQ)^2\cdot kP/Re$.
Then every class in $\H^1(S; kPe)$ has a representative $\zeta$ in $\Z^1(S; kPe)$
satisfying  
$$\zeta(S)\subseteq (J(kQ)^2+J(kR))\cdot kPe .$$
\end{Lemma}

\begin{proof}
Since $Q/R$ is characteristic in $P/R$ it follows that $Q$ is an $\langle e\rangle$-stable
normal subgroup in $P$. 
Thus this Lemma  is the special case of Lemma \ref{H1G-submodule} applied
with $G=S$, acting by conjugation on $M=kPe$, $U = (J(kQ)^2+J(kR))\cdot kPe$, 
and $V=J(kR) \cdot kPe$.
One checks that we have canonical isomorphisms
$M/V\cong$ $kP/Re$ and $U/V\cong $ $J(kQ)^2\cdot kP/Re$.
The result follows indeed from Lemma \ref{H1G-submodule}.
\end{proof}

\begin{Lemma}\label{L:Pe-6}
Let $P$ be a non-trivial finite $p$-group and let $e$ be a non-trivial 
$p'$-automorphism of $P$. 
Let $R$ be a normal subgroup of $P$ contained in $C_P(e)$. Let $Q$ be
the normal subgroup of $P$ containing $R$ such that $Q/R=Z(P/R)$. 
Suppose that $Q$ is not contained in $C_P(e)$.
Then 
$$Q/R = T/R\times F,$$ 
where $T=C_Q(e)$ and  where $F=[Q/R,\langle e\rangle]$ is non-trivial.
Restriction from $Q$ to $T$ indues an isomorphism
$$\H^1(Q;kP/Re) \cong \H^1(T;kP/Re)^{F} = \Z^1(T;kP/Re)^{F}. $$
Furthermore, for any $\tau\in \Z^1(Q; kP/R e)$  the following statements hold.
\begin{itemize}
\item[{\rm (i)}] 
We have 
$\tau(T)\subseteq (\sum_{g\in F}g)\cdot kP/Re \subseteq $ $J(kF)^2\cdot kP/Re
\subseteq $ $J(kQ)^2 \cdot kP/Re.$
\item[{\rm (ii)}]
We have  $\tau(T)=0$ if and only if $\tau\in$ $\B^1(Q; kP/Re)$.
\item[{\rm (iii)}]
We have $\tau(T) \subseteq$ $(kP/Re)^{Q}$.
\item[{\rm (iv)}]
 For any $a\in Q$ and any $t\in T$ we have 
 $\tau(at)=\tau(a) + \tau(t)$. 
 \item[{\rm (v)}] 
 For any $a\in Q$ and any $t\in T$ we have 
 $\tau((at)^p)=\tau(a^p)$.
\end{itemize}
\end{Lemma}

\begin{proof}
Since $e$ has order prime to $p$ and acts trivially on $R$ but non-trivially on $Q$,
it follows (e.g. from \cite[8.2.2]{KurzStell}) that $e$ acts non-trivially on $Q/R$. 
Since $Q/R=$ $Z(P/R)$ is abelian we have
$Q/R=$ $T/R\times F$ with $F \neq 1$. Lemma \ref{order3} implies that
$F \cong Q/T$ has order at least $3$,  as $e$ acts non-trivially on $F$.
We have an obvious  short exact sequence of finite $p$-groups
$$\xymatrix{ 1 \ar[r] & T\ar[r] & Q \ar[r] & F\ar[r] & 1}.$$
By \Cref{T:esH1}, we have an  exact sequence
$$0 \to \H^1(F;(kP/Re)^T)\to \H^1(Q;kP/Re) \to 
\H^1(T;kP/Re)^{F} \to \H^2(F;(kP/Re)^T)$$
Note that $Q$ acts trivially on $P/R$ by conjugation. In particular, 
$T$ acts trivially on $kP/Re$ and $F$ is contained in  $Z(P/R)$.
By \Cref{L:Pe-1}, $F$ acts freely on the set $P /Re$, with the action induced by conjugation
with $Q$ on $Pe$, and we have that $(kP/Re)^T=kP/Re$ is a projective $kF$-module.
Therefore, by Lemma \ref{L:Pe-55}, the previous exact sequence yields the
stated  isomorphism
$$\H^1(Q;kP/Re) \cong \H^1(T;kP/Re)^{F} = \Z^1(T;kP/Re)^{F}, $$
where the second equality comes from the fact that $T$ acts trivially on $kP/Re$.
This isomorphism is given by restriction, so given any  $\sigma\in$ $\Z^1(T; kP/Re)^F$
there is $\tau\in$ $\Z^1(Q; kP/Re)$ such that $\tau|_T=$ $\sigma$.
The injectivity of this isomorphism is equivalent to statement (ii).
	
Being $F$-stable implies that 
 for any $t\in T$ and  $a\in Q$, we have $\sigma(^at)=$ ${^{a} \sigma(t)}$.
 Notice that $[a,t]\in R\subseteq C_P(e)$, so $^at=[a,t]t=[^ea,^et]t=[^ea,t]t=^{^ea}t$. 
It follows that
$${^{ a}\sigma (t)}=\sigma(^{^ea}t)={^{^e{a}}\sigma(t)}. $$
That means $\sigma(t)={^{a (^e {a})^{-1}} \sigma(t)}$. 
By \Cref{L:Pe-2}, applied with $F$ instead of $P$, we have 
$\sigma(t)\in (kP/Re)^F$. Since $F$ is in the center of $P/R$, by \Cref{L:Pe-1}, 
we have 
$\sigma(T)\subseteq (\sum_{g\in F}g)\cdot kP/Re$.
The remaining inclusions in (i) follow from the fact that $e$ acts non-trivially on
$F$, and hence $F$ has
order at least $3$ (cf. Lemma \ref{order3}), 
whence $\soc(kF)\subseteq$ $J(kF)^2$.  
This proves (i). 

 Since $\tau$ is a $1$-cocycle,  for $a\in P^{s-1}$ and $t\in T$ we have
 $\tau(at)=\tau(a) + {^a\tau(t)}$.
 As mentioned before, we have  $Q/R\leq Z(P/R)$. Thus, by Lemma \ref{L:Pe-1}, the
 conjugation action of $a\in Q$ on $kP/Re$ is equal to the action
 given by left multiplication with $a({^ea})^{-1}$. Therefore, using (i), this action
 is trivial on $\sigma(T)=$ $\tau(T)$.  
 The action of $T$ on $kQ/R e$ is 
 trivial since $Q/R$ is abelian and $T$ centralises $e$.  
 Thus the conjugation action by $Q$ on $kP/R e$ is trivial on
 $\tau(T)$. This proves (iii). The previous equation yields therefore 
  $\tau(at)=\tau(a) + {^a\tau(t)}=\tau(a) + \tau(t)$, which  proves (iv).
  
 Lemma \ref{L:cocycle-2} (iv)  implies that
  $\tau((at)^p) =$ $\sum_{i=0}^{p-1} {^{(at)^i}\tau(at)}$.
  Since $T$ acts trivially, this is equal to
  $\sum_{i=0}^{p-1} {^{a^i}\tau(at)} =$ $\sum_{i=0}^{p-1} {^{a^i}(\tau(a)+\tau(t))}$.
Since $Q$ acts trivially on $\tau(T)$, it follows that this expression
  is equal to $\sum_{i=0}^{p-1} {^{a^i} \tau(a)} = \tau(a^p)$, using again
  Lemma \ref{L:cocycle-2} (iv).
  This shows (v).
  \end{proof}
  
  \begin{Lemma} \label{L:Pe-65}
  Let $P$ be a non-trivial finite $p$-group.
Let $e$ be a non-trivial $p'$-automorphism of $P$. 
Let $R$ be an $\langle e\rangle$-stable normal subgroup of $P$
contained in $C_P(e)$, and let $Q$ be
the normal subgroup of $P$ containing $R$ such that $Q/R=Z(P/R)$.
Suppose that $Q$ is not contained in $C_P(e)$.
Every class in $\H^1(P;kP/Re)$ has a representative $\tau \in \Z^1(P; kP/Re)$ 
such that 
$$\tau(Q)\subseteq  J(kQ)^2 \cdot kP/Re.  $$
\end{Lemma}
  
  \begin{proof}
  As before, set $T=C_Q(e)$ and $F=[Q/R, \langle e\rangle] \cong$
  $Q/T$. By the assumption, $R$ is a proper subgroup of $Q$. 
  Denote by $\bar a$ the image in $F$ of $a\in Q$
 under the canonical surjection $Q\to F$ with kernel $T$. 
 Since $e$ is a $p'$-automorphism
 acting non-trivially on $Q$, it follows from \cite[Theorem 5.3.10]{Gor} that
 $e$ acts non-trivially on the subgroup $\Omega_1(Q)$ generated by all
 elements of order at most $p$ in $Q$. Thus there is an element
 $c\in$ $Q$ of order $p$ such that $c\notin T$, hence such that
 $\langle c \rangle \cong$ $\langle \bar c\rangle$. Since $F$ is abelian, there
 is an element $a\in$ $Q$ such that $\bar a \notin \Phi(F)$ and such that
 $\bar c\in$ $\langle \bar a \rangle$.  Then $\langle a\rangle$ contains $ct$ for some
 $t\in T$. Note that the canonical surjection $\langle a\rangle \to \langle \bar  a\rangle$
 has kernel $\langle (ct)^p \rangle$.  Using Lemma \ref{L:Pe-6}(iv)  we get that
 $\tau((ct)^p) = \tau(c^p) = \tau(1) = 0$.
 Thus the restriction of $\tau$ to $\langle a\rangle$ vanishes on
 $\langle (ct)^p\rangle$, hence is inflated from
 $\langle a\rangle/\langle (ct)^p\rangle\cong$ $\langle \bar a \rangle$.
 The group $\langle \bar a \rangle$  acts freely on $kP/Re$, hence
 $H^1(\langle a\rangle; kP/P^se)=0$. Thus the restriction of $\tau$
 to $\langle a\rangle$ is a $1$-coboundary. That is, there is $m\in$ $kP/Re$ such that
 $\tau(x) = $ ${^xm}-m$ for all $x\in$ $\langle a\rangle$.
 Therefore, after replacing $\tau$ by $\tau-\mu$, where $\mu$ is the coboundary
 given by $\mu(x)={^xm}-m$ for all $x\in P$, we may assume that $\tau(x)=0$ 
 for all $x\in \langle a\rangle$. In particular, $\tau(a)=0=\tau(a^{-1})$.
 
 If $F$ is cyclic, hence generated by $\bar a$, then $Q=\langle a, T\rangle$,
 and hence the statement follows from Lemma \ref{L:Pe-6}  (i).
 Assume that $F$ is not cyclic. Thus $F$ has rank at least $2$ and order at least
 $p^2$, so $\soc(kF)\subseteq$ $J(kF)^{2p-1}\subseteq$ $J(kF)^5$.
 Using the fact that $kF$ is isomorphic to a truncated polynomial ring ,
 this implies that the product of any
 two elements in $J(kF)\setminus J(kF)^2$ is not contained in $\soc(kF)$. 
 Let $y\in Q$. 
 
 Calculating in the abelian group $Q/R \cong$ $T/R\times F$, we have
 $[\bar a(^e{\bar a})^{-1}, \bar y]=1$ for all $y\in Q$.
 Using that $\tau(a)=0=\tau(a^{-1})$ we get from Lemma \ref{L:cocycle-2} (iii) that 
 $$\tau(aya^{-1}y^{-1}) = 
 {^a{\tau(y)}}  +  {^{aya^{-1}}\tau(y^{-1})}.$$
 Clearly $aya^{-1}=$ $(aya^{-1}y^{-1})y$. 
 Since $aya^{-1}y^{-1}\in R$ acts trivially by
 conjugation on $kP/Re$ it follows that
  $$\tau(aya^{-1}y^{-1}) = {^a\tau(y)}  +  {^y\tau(y^{-1})} = 
  {^a\tau(y)} - \tau(y),$$
  where we have used  Lemma \ref{L:cocycle-2} (ii) for the last equality.
 The conjugation action by $a$ is
 equal to left multiplication by $\bar a (^e\bar e)^{-1}$ 
 (cf. Lemma \ref{L:Pe-1} applied with the central subgroup $Q/R$ of
 $P/R$), and hence we get that
   $$\tau(aya^{-1}y^{-1}) = (\bar a (^e\bar a)^{-1} - 1)\tau(y).$$
 Now $\bar a(^e\bar a)^{-1}$ does not belong to $\Phi(F)$ since $\bar a$ does not,
 and hence $\bar a(^e\bar a)^{-1} - 1\in$ $J(kF)\setminus J(kF)^2$,
 so not contained in $\soc(kF)$ as $F$ has order at least $3$.
 Since $aya^{-1}y^{-1}\in$ $R\leq$ $T$ it follows from Lemma \ref{L:Pe-6} (i) that 
 $\tau(aya^{-1}y^{-1})\in$ $\soc(kF)\cdot kP/Re$. 
 Since $F$ has rank at least $2$, this implies that $\tau(y)\in$
 $J(kF)^2\cdot kP/Re$, whence the Lemma.
\end{proof}

The next Proposition shows that Theorem \ref{T:kPe-1} holds whenever $e$
acts trivially on the commutator subgroup $P'$.

\begin{Proposition} \label{kPe-Prop-2}
Let $P$ be a finite $p$-group and $e$ a non-trivial  $p'$-automorphism of $P$.
Suppose that $e$ acts trivially on the commutator subgroup $P'$ of $P$.
Consider $kPe$ as $kP$-module with  $P$ acting by conjugation.
Then every class in $\H^1(P; kPe)$ has a representative $\zeta$ with
$\zeta(P)\subseteq$ $J(kP)^2e$.
\end{Proposition}

\begin{proof}
Since $e$ acts non-trivially on $P$ and $P'\leq \Phi(P)$, 
 it follows from \cite[Theorem 5.1.4]{Gor}  that $e$ acts non-trivially on 
$P/P'$.  Since $e$ acts trivially on $P'$ and since $P/P'$ is
abelian, it follows from Lemma \ref{L:Pe-65} 
that every class in $\H^1(P; kP/P'e)$ has a representative with
image contained in $J(kP)^2\cdot kP/P'e$. Thus, by Lemma \ref{L:Pe-5},
every class in $\H^1(P; kPe)$ has a representative in $\Z^1(P; kPe)$ with
image contained in $(J(kP)^2+J(kP'))\cdot kP=$ $J(kP)^2e$, where the
last equality uses Lemma \ref{L-Jsquare}
\end{proof}

\begin{proof}[{Proof of Theorem \ref{T:kPe-1}}]
By Proposition \ref{kPe-Prop-1} we may assume that $e$ acts trivially on $Z(P)$,
and by Proposition \ref{kPe-Prop-2} we may assume that $e$ acts non-trivially 
on $P'$, hence non-trivially on $\Phi(P)$. In particular, $C_{\Phi(P)}(e)$ is a proper 
subgroup of $\Phi(P)$.
It follows therefore from  Lemma \ref{L:Pe-71} that there are non-trivial characteristic 
subgroups $Q$, $R$ of $P$ both contained in $\Phi(P)$ 
such that  $R\leq Q$, such that $Q/R\leq$ $Z(P/R)$, and such that
$e$ acts trivially on $R$ but non-trivially on $Q$.  In particular,$Q/R$ is non-trivial,
abelian, and $e$ acts non-trivially on this quotient.

By earlier conventions, we denote by $P/Re$ the $P/R$-coset of $e$ in 
$P/R\rtimes \langle e \rangle$; this is a $P$-$P$-biset since $R$ is normal in $P$
and both $P$ and $R$ are normalised by $e$. We use the obvious
analogue of this notation  with $Q$ instead of $R$.
By  Lemma \ref{L-Jsquare}
we have $J(kR)\subseteq$ $J(kP)^2$.  Lemma \ref{L:Pe-65}
implies that  every class in $\H^1(P;kP/Re)$ has a representative 
$\tau \in \Z^1(P; kP/Re)$  such that 
$\tau(Q)\subseteq  J(kQ)^2 \cdot kP/Re$.
It follows from this and Lemma \ref{L:Pe-5} that
every class in $\H^1(P; kPe)$ has a representative $\zeta$ in $\Z^1(P; kPe)$
satisfying  $\zeta(Q)\subseteq (J(kQ)^2+J(kR))\cdot kPe$.
In particular, $\zeta(Q)\subseteq$ $J(kQ)\cdot kPe$. By Lemma \ref{H1-induced},  
$\zeta$ induces
uniquely a cocycle $\bar\zeta$ in $\Z^1(P/Q: kP/Qe)$. 
Since $Q\leq$ $\Phi(P)$ it follows that $e$ acts non-trivially on $P/Q$.

Arguing by induction, there is a cocycle $\sigma\in$ $\Z^1(P/Q;  kP/Qe)$
representing the same class as $\bar\zeta$ and satisfying $\sigma(P/Q)\subseteq$
$J(kP)^2\cdot kP/Qe$.  Thus, if we inflate $\bar\zeta$ and $\sigma$ to
cocycles in $\Z^1(P; kP/Qe)$, the resulting cocycles represent the same
class in $\H^1(P; kP/Qe)$. Denoting the inflations of $\bar\zeta$ and $\sigma$
by the same letters, it follows that $\sigma(P)\subseteq$ $J(kP)^2\cdot kP/Qe$.
Now $\bar\zeta = \sigma + \beta$ 
for some $\beta\in$ $\B^1(P; kP/Qe)$. The coboundary  $\beta$  lifts to a
coboundary $\alpha\in$  $\B^1(P; kPe)$ via the map $\B^1(P;kPe)\to$ $\B^1(P; kP/Qe)$
induced by the canonical map $kPe\to kP/Qe$. 
Thus $\zeta'=$ $\zeta- \alpha$ lifts $\sigma$, and represents the same class as $\zeta$.
Since $\sigma$ has image contained in $J(kP)^2 \cdot kP/Qe$ it follows that
$\zeta'$ has image contained in the inverse image of this space in $kPe$, and that
is $(J(kP)^2 + J(kQ))\cdot kPe$. Since $Q\leq \Phi(P)$ it follows again from
Lemma \ref{L-Jsquare} that $\zeta'(P)\subseteq$ $J(kP)^2e$.
This completes the proof  of Theorem \ref{T:kPe-1}.
\end{proof}

\begin{Remark}
Let $1=P^{m}< P^{m-1}<\cdots <P^1=P'<P^0=P$ be the lower central series of $P$.
Assume $e$ acts trivially on $Z(P)$. Then there exists an integer $t$ with $m-1< t\leq 1$ 
such that $P^t\leq C_P(e)$ but $P^{t-1}\not\leq C_P(e)$. Another way to present the 
proof above for Theorem \ref{T:kPe-1}, is to do induction on $t$. If $t=1$, then the 
theorem follows by Proposition \ref{kPe-Prop-2}. Suppose $t\geq 2$. Then by Lemma 
\ref{L:Pe-65} and \ref{L:Pe-5}, every class in $\H^1(P; kPe)$ has a representative 
$\tau\in \Z^1(P; kPe)$ such that 
$$\tau (P^{t-1})\subseteq (J(kP^{t-1})^2+J(kP^t))kPe.$$
In particular, $\tau$ induces uniquely a cocycle $\bar \tau$ in 
$\Z^1(P/P^{t-1};kP/P^{t-1}e)$.  The rest of the induction is similar to the proof above. 
\end{Remark}

\section{On  $\HH^1(k_\alpha(P\rtimes E))$ and 
proofs of the Theorems \ref{thm2}, \ref{thm3}, \ref{thm4}, \ref{thm5}}
\label{main-Section} 

The following Theorem extends \cite[Theorem 1.3.(i)]{LW25}, and it includes a
proof of Theorem \ref{thm2}. This Theorem 
will be needed in the proof of Theorem \ref{thm1} in the next Section
Let $k$ be a field of prime characteristic $p$.
Given a finite $p$-group $P$,
a $p'$-subgroup $E$ of $\Aut(P)$, and a $2$-cocycle $\alpha\in$ $\Z^2(E; k^\times)$
inflated to $P\rtimes E$, we will make use of
the fact that $\HH^1(kP)^E$ can be canonically identified with a Lie subalgebra of
$\HH^1(k_\alpha(P\rtimes E))$; see e.g. \cite[Lemma 5.1]{LW25}. 
Since $E$ is a $p'$-group, we have $J(k_\alpha(P\rtimes E))=$ 
$J(kP)\cdot k_\alpha(P\rtimes E)=$ $k_\alpha(P\rtimes E)\cdot J(kP)$.
We refer to 
\cite[Section 5]{LW25} for more details regarding calculations in twisted
group algebras. We use without further comment the fact that a surjective
$E$-stable map between $k$-vector spaces remains surjective upon taking $E$-fixed
points; as briefly mentioned in \cite[Lemma 2.9]{LW25}. This is relevant
in the proof of the following Theorem when it comes to  lifting $E$-stable
cohomology classes to $E$-stable derivations (see e.g. \cite[Lemma 2.9]{LW25}). 
We use notation introduced in
Proposition \ref{rad-Prop}; that is,  for $m\geq 1$, we denote by 
$\Der_m(k_\alpha(P\rtimes E))$ the subspace of $\Der(k_\alpha(P\rtimes E))$ 
consisting of all derivations $f$ on $k_\alpha(P\rtimes E)$  such that 
$\Im(f) \subseteq$ $J(k_\alpha(P\rtimes E))^m=$ $J(kP)^m \cdot k_\alpha(P\rtimes E)$.

\begin{Theorem} \label{T:kPE}
Let $P$ be a non-trivial finite  $p$-group and $E$ a $p'$-subgroup of $\Aut(P)$. 
Let $\alpha\in$ $\Z^2(E;k^\times)$ inflated to $P\rtimes E$ via the canonical 
surjection $P\rtimes E\to E$.  For $m$ a positive integer, denote by $D_m$ the image
in $\HH^1(k_\alpha(P\rtimes E))$ of $\Der_m(k_\alpha(P\rtimes E))$.
\begin{itemize}
\item[{\rm (i)}] 
We have 
\[  \HH^1(k_\alpha(P\rtimes E))= \HH^1(kP)^E+ \CN . \]
for some Lie ideal $\CN$ of $\HH^1(k_\alpha(P\rtimes E))$ contained in $D_2$.
In particular, $\CN$ is nilpotent. 

\item[{\rm (ii)}]
The Lie algebra $\HH^1(k_\alpha(P\rtimes E))$ is solvable if and only if
$\HH^1(kP)^E$ is solvable.

\item[{\rm (iii)}]
If the Lie algebra $\HH^1(k_\alpha(P\rtimes E))$ is simple, then
$\HH^1(k_\alpha(P \rtimes E))=$ $\HH^1(kP)^E$.

\item[{\rm (iv)}]
If $[P,E]=P$, then $D_1= \HH^1(k_\alpha(P\rtimes E))$,
and $D_2$ is an ideal in the Lie algebra $\HH^1(k_\alpha(P\rtimes E))$.
\end{itemize}
\end{Theorem} 

\begin{proof}
By  \cite[Lemma 5.1]{LW25}, we have a canonical isomorphism
$$\HH^1(k_\alpha (P\rtimes E))\cong
\HH^1(kP)^E\oplus (\bigoplus_{e\in E\setminus\{1\}} \H^1(P;kP\cdot e))^E.
$$
We note that we have a canonical isomorphism $\H^1(P; kP\cdot e)\cong$
$\HH^1(kP; kP\cdot e)$. 

Let $e\in E\setminus \{1\}$. By Theorem \ref{T:kPe-1},
every class in $H^1(P; kP\cdot e)$ has
a representative with image in $J(kP)^2 \cdot e$, and hence the corresponding
derivation $kP\to $ $kP\cdot e$ has image in $J(kP)^2\cdot e$, by Proposition 
 \ref{HH1H1-Prop} (iii). 
This shows that the space
$(\bigoplus_{e\in E\setminus\{1\}} \H^1(P;kP\cdot e))^E$ is contained in $D_2$.
This space is a Lie module for the action of the Lie subalgebra $\HH^1(kP)^E$.
Indeed, every
class in $\HH^1(k_\alpha(P\rtimes E))$ is represented by a derivation which
vanishes on the separable subalgebra $k_\alpha E$ 
(see e.g. \cite[Lemma 2.4]{LiRusolvable}), hence which is in particular
a $k_\alpha E$-endomorphism by the Leibniz rule.
 Any such endomorphism of $k_\alpha(P\rtimes E)$
is determined by its restriction to $kP$. 

By \cite[Proposition 5.2]{LW25}, the
image of $\HH^1(kP)^E$ in $\HH^1(k_\alpha(P\rtimes E))$ is represented by 
$k_\alpha E$-endomorphisms which map $kP$ to   $kP$, and one easily checks,
calculating in $\End_{k_\alpha E}(k_\alpha(P\rtimes E))$ that these preserve the
endomorphisms mapping $kP$ to 
$\bigoplus_{e\in E\setminus\{1\}} \HH^1(kP;kP\cdot e))$ via composition and
precomposition of endomorphisms.

It follows that the Lie ideal $\CN$ in $\HH^1(k_\alpha(P\rtimes E))$
generated by the space corresponding to the sum
 $\bigoplus_{e\in E\setminus\{1\}} \H^1(P;kP\cdot e))^E$
is still contained in $D_2$.
In particular, $\CN$ is nilpotent (cf. \cite[Proposition 2.3]{LW25} or
Proposition  \ref{rad-Prop} above).
This shows (i).

Statement (ii) follows from (i) and the fact that $\CN$ is a nilpotent
ideal, and statement (iii) is an immediate consequence of (i) as well.

Assume now that $[P,E]=P$. 
Since $E$ is a $p'$-group, any $E$-stable class in $\HH^1(kP)$ has an $E$-stable 
representative in $\Der(kP)$ (see e.g. \cite[Lemma 2.9]{LW25}).
Using the hypothesis $[P,E]=P$, it follows from \cite[Corollary 5.6]{LW25}   
that  every $E$-stable derivation
on $kP$ has image contained in $J(kP)$. It follows that every class in
$\HH^1(kP)^E$ is represented by derivations with image in $J(kP)$.
 Proposition \ref{rad-Prop} implies that  $D_2$ is a Lie ideal.
This shows (iv).
\end{proof}

\begin{proof}[{Proof of Theorem \ref{thm2}}]
Theorem \ref{thm2} is an immediate consequence  of 
in Theorem \ref{T:kPE} (ii) combined with K\"ulshammer's Theorem
\ref{Kuenormal}.
\end{proof}

\begin{proof}[{Proof of Theorem \ref{thm3}}]
By \cite[Theorem 1.4]{BKL2021},
if $P$ is extraspecial, then $\HH^1(kP)$ is a solvable Lie algebra.
Thus $\HH^1(kP)^E$ is a solvable Lie algebra for any $p'$-subgroup 
$E$ of $\Aut(P)$. The result follows from Theorem \ref{thm2}.
\end{proof}

\begin{proof}[{Proof of Theorem \ref{thm4}}]
With the notation above, by Theorem \ref{T:kPE}, the Lie algebra
$\HH^1(k_\alpha(P\rtimes E))$ is solvable if and only if $\HH^1(kP)^E$
is solvable. Set $\bar P=P/\Phi(P)$. The kernel of the canonical map
$\HH^1(kP)^E\to$ $\HH^1(k\bar P)^E$ is a nilpotent Lie ideal (cf. Lemma
\ref{PPhiP-Lemma} and Proposition \ref{rad-Prop}).
Thus $\HH^1(kP)^E$ is solvable if $\HH^1(k\bar P)^E$ is solvable, and now the 
first result follows from \cite[Theorem  1.2]{LW25}. 
Now assume $P$ is abelian. We claim that when $p$ is odd, if $[P,E]\neq P$, then 
$\HH^1(kP)^E$ is not solvable. Then the second statement follows again by 
\cite[Theorem  1.2]{LW25}. Suppose that $p$ is odd and $Q=[P,E]\neq P$.
 Let $T=C_P(E)$. We have $P=T\times Q$ and $T\neq 1$. By the K\"unneth formula, 
 we have $\HH^1(kP) \cong \HH^1(kT)\otimes kQ\oplus  kT\otimes \HH^1(kQ)$ and 
 $\HH^1(kT)$ is a nonzero Lie subalgebra of $\HH^1(kP)$. Then 
 $\HH^1(kT)^E=\HH^1(kT)$ is a nonzero Lie subalgebra of $\HH^1(kP)^E$. By 
 \cite[Theorem 1.2, Example 5.1]{Antrobus}, $\HH^1(kT)$ is not solvable. Thus, 
 $\HH^1(kP)^E$ is not solvable.
\end{proof}

While Theorem \ref{thm4} (i) gives a sufficient condition for the solvability of 
$\HH^1(kP)^E$, the following proposition gives a necessary condition.

\begin{Proposition}
Let $P$ be a finite $p$-group, and let $E$ be a $p'$-subgroup of $\Aut(P)$. 
If $\HH^1(kP)^E$ is solvable then $Z(P)^E\leq \Phi(P)$. 
\end{Proposition}

\begin{proof}
Suppose that there exists $g\in Z(P)^E$ and $g\not\in\Phi(P)$. By Proposition 
\ref{HomPkE}, we have $\Hom(P,k)^E\neq 0$. In particular, 
$\Hom (\langle g\rangle,k)\leq \Hom(P,k)^E\neq 0$. By the centraliser decomposition, 
we have that $kZ(P)\otimes \Hom(P,k)$ is a Lie subalgebra of $\HH^1(kP)$, where for 
$z\in Z(P)$, $\zeta\in\Hom(P,k)$ and $g \in P$,  $(z\otimes \zeta)(g)=\zeta(g)gz$. Upon 
taking $E$-fixed points, we have 
$$kZ(P)^E\otimes \Hom(P,k)^E\leq (kZ(P)\otimes \Hom(P,k))^E\leq \HH^1(kP)^E.$$ 
On the other hand, $k\langle g\rangle\otimes \Hom (\langle g\rangle,k)\cong$
$\HH^1(k\langle g\rangle)$ is a Lie subalgebra of $Z(P)^E\otimes \Hom(P,k)^E$, 
thus, a Lie subalgebra of $\HH^1(kP)^E$. By \cite[Theorem 1.2, Example 5.1]{Antrobus} 
again, $\HH^1(k\langle g\rangle)$ is not solvable. Then $\HH^1(kP)^E$ is not solvable.
\end{proof}

\begin{proof}[{Proof of Theorem \ref{thm5}}]
By a result of Chuang and Rouquier \cite[Theorem 7.2]{ChRou},
 blocks of symmetric groups with
isomorphic defect groups are derived equivalent. As a consequence of a result of
Chuang and Kessar \cite[Theorem 2]{ChKe},  amongst all blocks of symmetric 
groups with a given
abelian defect group, there is always one block that is in derived
equivalent to its Brauer correspondent (or equivalently, which satisfies Brou\'e's
abelian defect conjecture). Thus it suffices to show that $\HH^1(k(P\rtimes E))$
is solvable, with $P$ an elementary abelian Sylow
$p$-subgroup of $S_n$ for some $n$ such that $n=pw$ for some positive
integer $w\leq p-1$  and inertial quotient $E\cong N_{S_n}(P)/P$. 
By Theorem \ref{T:kPE} this is equivalent to showing that $\HH^1(kP)^E$
is solvable. 
It is well-known that then $E\cong$ $C_{p-1}\wr S_w$, and that 
$P/\Phi(P)$ is simple as an $\Fp E$-module. The result follows from
\cite[Theorem 1.2]{LW25}.
\end{proof}

As mentioned before, and with the notation above,  
by \cite[Proposition 5.2]{KeLiFrob}, if $E$ acts freely on 
$P\setminus\{1\}$, then $\HH^1(k_\alpha(P\rtimes E))=$ $\HH^1(kP)^E$.
The following Proposition shows  that if $\alpha$ is trivial, then the converse holds.

\begin{Proposition} \label{Frob-Prop}
Let $P$ be a  finite $p$-group and let $E$ be a $p'$-subgroup of $\Aut(P)$. 
We have $\HH^1(k(P\rtimes E))=$ $\HH^1(kP)^E$ if and only if $E$ acts
freely on $P\setminus \{1\}$.
\end{Proposition}

\begin{proof}
Suppose that $E$ does not act freely on $P\setminus \{1\}$. That is, there
exists $e\in E$, $e\neq 1$, such that $C_P(e)\neq 1$. Let $g\in$ $C_P(e)$
such that $g\not\in \Phi(C_P(e))$.  Set $G=P\rtimes E$ and $C=C_G(ge)$.
Since $g$, $e$ have coprime order and commute with each other, we have
$C=C_G(g)\cap C_G(e)$. The group $R=C_P(ge)=C_P(e)\cap C$ contains $g$,
hence is non-trivial. Note that $C_P(e)$ is the unique Sulow $p$-subgroup of
$C_G(e)$. Hence $R$ is the unique Sylow $p$-subgroup of $C$. 
By Theorem \ref{T:esH1},  we have an exact sequence of the form
$$\xymatrix{0 \ar[r] & \H^1(C/R; k) \ar[r] & \H^1(C; k) \ar[r] & \H^1(R;k)^{C/R}
\ar[r] & H^1(C/R; k)}.$$
Now $C/R$ has order prime to $p$, so $\H^n(C/R; k)=0$ for $n>0$.
Thus the above exact sequence yields an isomorphism
$$\H^1(C;k) \cong H^1(R; k)^{C/R}.$$
Now $\Phi(R)\leq \Phi(C_P(e))$, so $g\in$ $R\setminus \Phi(R)$, and $C/R$ acts
trivially on $g$. Thus $\H^1(R;k)^{C/R}=$ $\Hom(R,k)^{C/R}\neq \{0\}$. 
It follows that $\H^1(C;k)\neq 0$. Thus there is a nonzero derivation on 
$k(P\rtimes E)$ corresponding to $ge$ in the centraliser decomposition with 
image in the subspace of $k(P\rtimes E)$ spanned by the conjugacy class of
$ge$, hence in the complement $(\oplus_{e\in E\setminus \{1\}}\ \HH^1(kP; kPe))^E$
of $\HH^1(kP)^E$ in the decomposition given in \cite[Lemma 5.1]{LW25}.
This shows that if $E$ does not act freely on $P\setminus \{1\}$, then
$\HH^1(k(P\rtimes E))\neq \HH^1(kP)^E$. The converse is clear by 
\cite[Proposition 5.2]{KeLiFrob}.
\end{proof}

We do not know whether the `only if' part of Proposition \ref{Frob-Prop} 
holds more generally  for $k_\alpha(P\rtimes E)$ with $\alpha$ representing  
an arbitrary   class in $\H^2(E;k^\times)$.

\section{Proof of Theorem \ref{thm1} and its Corollaries}
\label{thm1-Section}.

The main purpose of this Section is to prove the following result which is then 
shown to imply  Theorem \ref{thm1}. Let $k$ be a field of prime characteristic $p$.

\begin{Theorem} \label{thm1-PE}
Let $P$ be a finite $p$-group  and $E$ a $p'$-subgroup of $\Aut(P)$. 
Let $\alpha\in \Z^2(E;k^\times)$ inflated to $P\rtimes E$ via the canonical
surjection $P\rtimes E\to$ $E$. If $\HH^1(k_\alpha (P\rtimes E))$ is simple, 
then $E$ is trivial and $P$ is elementary abelian of order at least $3$.
\end{Theorem}

The first step in the proof of Theorem \ref{thm1-PE} is the 
following result which will be  used to play back the proof of Theorem \ref{thm1-PE}
to the case where $P$ is abelian.

\begin{Proposition} \label{HH1kPE-nonsimple}
Let $P$ be a finite non-abelian $p$-group, and let $E$ be a $p'$-subgroup of
$\Aut(P)$. Set $\bar P=P/\Phi(P)$. The canonical homomorphism of Lie algebra  
$\HH^1(kP)^E\to\HH^1(k\bar P)^E$
has a nonzero kernel. In particular,  the Lie algebra $\HH^1(kP)^E$ is not simple.
\end{Proposition}

\begin{proof}
Since $P$ is not abelian, the group $\Phi(P)Z(P)$ is a proper subgroup of $P$.
Let $g\in P\setminus \Phi(P)Z(P)$. Then $Q=C_P(g)$ is a proper subgroup of $P$. Let $K$ be
the stabiliser in $E$ of the conjugacy class, denoted $C$, of $g$ in $P$.
Set $\bar P =$ $P/\Phi(P)$, and denote by $\bar g$ the  image of $g$ in $\bar P$. 
Note that $\bar g$ has order $p$ in $\bar P$ and $K$ acts trivially on $\bar g$.
Note that we have an obvious isomorphism of permutation $kP$-modules
$\Ind^P_Q(k)\cong$ $kC$, sending $u\ten 1$ to ${^u{g}}$, where $u\in P$.
Using that $E$, hence $K$,  has order prime to $p$
it follows that $\langle \bar g\rangle$ has a $K$-stable complement in $\bar P$,
and hence there is an $K$-stable group homomorphism $\bar P\to k$ sending
$\bar g$ to $1$. Inflating this to $P$ yields a $K$-stable group homomorphism
$\lambda : P\to k$ sending $g$ to $1$, with kernel a maximal $K$-stable subgroup
of $P$. Set $\tau = \Res^P_Q(\lambda)$. Since $g\in Q$ it follows that $\tau : Q\to k$
is a nonzero group homomorphism, hence represents a nonzero element in
$\H^1(Q;k)$. Through the Eckmann-Shapiro isomorphism
$\H^1(Q;k)\cong$ $\H^1(P; \Ind^P_Q(k))$  this yields a nonzero $\zeta$  in
$\H^1(P; \Ind^P_Q(k))$,  hence in $\H^1(P; kC)$, and so in $\H^1(P; kP)$ via the
inclusion $kC\to kP$. We denote this abusively again by $\zeta$. 
The derivation on $kP$ corresponding to $\zeta$ has
image contained in the subspace $kC$. 
Under the canonical map $P\to \bar P$, the
conjugacy class $C$ of $g$ is mapped to $\{\bar g\}$, hence induces the canonical
map $\Ind^P_Q(k) \to k$, where we identify $\Ind^P_Q(k)=kC$ and
$k=k\cdot \bar g$.  By Lemma \ref{tr1-Lemma}, 
the composition $H^1(Q;k)\cong$ $H^1(P;\Ind^P_Q(k))\to$
$\H^1(P;k)$ induced by the Eckmann-Shapiro isomorphism and the
canonical surjection $\Ind^P_Q(k)\to k$ is equal to the transfer $\Tr^P_Q(\tau)$.
But $\tau=$ $\Res^P_Q(\lambda)$, and hence $\Tr^P_Q(\tau)=0$. This shows
that the nonzero element $\zeta$, regarded as an element in $H^1(P;kP)$,
is in the kernel of the canonical map $\H^1(P; kP)\to$ $\H^1(P; k\bar P)=$ 
$\H^1(\bar P; k\bar P)$. Now by construction, $\zeta$ is stable under the
action of the stabiliser $K$ in $E$ of the conjugacy class $C$ of $g$ in $P$.
It follows that $\Tr^E_K(\zeta)$ remains a nonzero  class in $\H^1(P;kP)$,
since no element of $E$ outside of $K$ stabilises the conjugacy class $C$. 
Moreover, the class $\Tr^E_K(\zeta)$ is $E$-stable, and is still in the kernel of the
canonical map $\HH^1(P; kP)\to$ $H^1(\bar P; k\bar P)$. With the standard
identification $\HH^1(kP)\cong$ $H^1(P;kP)$ and its analogue for $\bar P$,
this shows the first statement. 

By Lemma \ref{L-Jsquare} or Lemma \ref{PPhiP-Lemma}. the kernel of the canonical 
map  $\HH^1(kP)\to$ $\HH^1(k\bar P)$
is represented by derivations with image contained in $J(kP)^2$, and so this kernel
is a nilpotent Lie ideal. It follows that if $\HH^1(kP)^E$ is simple, then it is not
contained in the kernel of this map, and hence that the map $\HH^1(kP)^E\to$
$\HH^1(k\bar P)^E$ is injective. That, however,  contradicts the first statement,
completing the proof.
\end{proof}

One of the main steps in the proof of Theorem \ref{thm1} in the case where $P$
is abelian  is the following Lemma, which relies on Theorem \ref{T:kPE} (ii). 

\begin{Lemma}\label{L:HH1notsimple}
Let $P$ be a non-trivial abelian $p$-group and $E$ a non-trivial $p'$-subgroup of $\Aut(P)$. 
Let $\alpha\in \Z^2(E;k^\times)$ inflated to $P\rtimes E$ via the canonical surjection 
$P\rtimes E\to$ $E$. Suppose that $[P,E]<P$. Then 
$\HH^1(k_\alpha(P\rtimes E))$ is not simple. 
\end{Lemma}

\begin{proof}
Since $P$ is abelian, it follows that  $P=Q\times R$ with $Q=[P,E]$ and $R=C_P(E)$.
The hypotheses imply that 
$Q$, $R$ are both nontrivial. Thus we have $G=R\times (Q\rtimes E)$, 
$kG\cong kR\otimes k_\alpha (Q\rtimes E)$. By the K\"unneth formula we have
\[\HH^1(k_\alpha (P\rtimes E))\cong Z(kR)\otimes 
\HH^1(k_\alpha (Q\rtimes E))\oplus \HH^1(kR)\otimes Z(k_\alpha (Q\rtimes E)).
\]
By Theorem \ref{T:kPE} (ii), the Lie algebra homomorphism 
$\Der_1(k_\alpha (Q\rtimes E))\to \HH^1(k_\alpha (Q\rtimes E))$ 
is surjective and $\HH^1(kR)$ is non-zero since $R$ is nontrivial. 
By Lemma \ref{KuennethRad2}, $\HH^1(k_\alpha (P\rtimes E))$ is not simple.
\end{proof}

\begin{Lemma}\label{L:HH1kGsimple}
Let $P$ be a non-trivial abelian $p$-group and $E$ a $p'$-subgroup of $\Aut(P)$. 
Let $\alpha\in \Z^2(E;k^\times)$ inflated to $P\rtimes E$ via the canonical
surjection $P\rtimes E\to$ $E$. If $\HH^1(k_\alpha (P\rtimes E))$ is simple, 
then $E$ is trivial. 
\end{Lemma}

\begin{proof}
By Lemma \ref{L:HH1notsimple}, this result  holds when $[P,E]<P$. 
Suppose $[P,E]=P$. Since $P$ is assumed to be non-trivial, it follows that 
$E$ is non-trivial. Let $D_i$ denote the image of 
$\Der_i(k_\alpha (P\rtimes E))$ in $\HH^1(k_\alpha (P\rtimes E))$ for 
positive integer $i$. By Theorem \ref{T:kPE}, we have
$$\HH^1(k_\alpha (P\rtimes E))=D_1= \HH^1(kP)^E+ D_2
$$
If $D_2$ is nonzero, by Proposition \ref{rad-Prop}, $D_2$ is a nilpotent ideal in $D_1$ 
and therefore $\HH^1(k_\alpha (P\rtimes E))$ is not simple. If $D_2$ is zero, then 
$\HH^1(k_\alpha (P\rtimes E))= \HH^1(kP)^E$. By Lemma \ref{lemglk}, we have
\[\HH^1(k_\alpha (P\rtimes E))= \HH^1(kP)^E=\Der(kP)^E,
\]
and there is a surjective Lie algebra homomorphism from $\Der(kP)^E$ to a 
product of general linear Lie algebras which is not simple. Thus, 
$\HH^1(k_\alpha (P\rtimes E))$ is not simple.
\end{proof}

\begin{proof}[{Proof of Theorem \ref{thm1-PE}}]
Assume that the Lie algebra  $\HH^1(k_\alpha(P\rtimes E))$
is simple. Then the nilpotent ideal $\CN$ from Theorem \ref{T:kPE} (i) is 
necessarily zero. But then $\HH^1(P_\alpha(P\rtimes E))=$ $\HH^1(kP)^E$.
Proposition \ref{HH1kPE-nonsimple} implies that  $P$ is  abelian.
Since  $\HH^1(k_\alpha (P\rtimes E))$ is simple,  it follows from  \Cref{L:HH1kGsimple} 
that $E$ must be trivial. But then  $\HH^1(k_\alpha (P\rtimes E))=\HH^1(kP)$, 
and hence $P$ is elementary abelian of order at least $3$ by 
\cite[Theorem 1]{Jacobson} and \cite[Proposition 2.5]{LiRusimple}. The result follows.
\end{proof}

\begin{proof}[{Proof of Theorem \ref{thm1}}]
Theorem \ref{thm1} is an immediate consequence of Theorem \ref{thm1-PE}
combined with K\"ulshammer's Theorem \ref{Kuenormal} and the fact that
$\HH^1(B)$, as a Lie algebra, is invariant under Morita equivalences.
\end{proof}

\begin{proof}[{Proof of Corollary \ref{cor1}}]
By Theorem \ref{thm1} applied to the Brauer correspondent $C$ of $B$,
the Lie algebra $\HH^1(C)$ is simple if and only if $C$ is nilpotent and
$P$ is elementary abelian of order at least $3$.  By the assumptions,
there is a stable equivalence of Morita type between $B$ and $C$.
It follows from \cite[Theorem 10.7]{KLZ} that $\HH^1(B)\cong$ $\HH^1(C)$
as Lie algebras. 
Since $P$ is abelian, it follows that $C$ is nilpotent if and only if $B$ is
nilpotent (see e.g. \cite[Proposition 8.11.3]{LiBookII}),  and in that case, 
both blocks are Morita equivalent to $kP$. The statement follows.
\end{proof}

\begin{proof}[{Proof of Corollary \ref{cor2}}]
By a result of Puig (cf. \cite[Theorem 10.5.1]{LiBookII}), 
there is a stable equivalence of Morita type between $B$ and its Brauer
correspondent $C$. Since $P$ is abelian, the blocks $B$ and $C$ are
nilpotent if and only if $E$ is trivial. Thus Corollary \ref{cor2} follows from
Corollary \ref{cor1}.
\end{proof}

\section{Appendix: further remarks on extraspecial $p$-groups}

Let $k$ be a field of prime characteristic $p$.
In the proof of Theorem \ref{thm3} we have made use of the fact, from
\cite[Theorem 1.4, Theorem 11.6]{BKL2021}, that if $P$ is a finite extraspecial 
$p$-group, then  $\HH^1(kP)$ is solvable.  The proof given in \cite{BKL2021} makes
use of some properties of the BV operator described in \cite[Theorem 1.1]{BKL2021}. 
We give an alternative proof of this  fact for $|P|\geq p^5$,
with some minor improvement  on the structure of $\HH^1(kP)$ in that case,
with starting point the following elementary group theoretic property of
extraspecial  $p$-groups.  See \cite{Antrobus} for more general results on the
solvability of $\HH^1(kP)$ for odd order $p$-groups $P$.

\begin{Proposition} \label{nonabelian-centralisers}
Let $P$ be an extraspecial finite $p$-group of order at least $p^5$.
Set $Z=Z(P)$, and let $x\in P\setminus Z$. Then $C_P(x)$ is a nonabelian
normal subgroup of index $p$ in $P$ containing $Z$.
\end{Proposition}

\begin{proof}
Since $Z$ is the commutator subgroup of $P$ it follows that
the conjugacy class of $x$ in $P$ is $Zx$, hence has length $p$, and thus $Q=C_P(x)$ has
index $p$ in $P$, contains $Z$,and is normal in $P$. Let $y\in P\setminus Q$, so
$P=\langle Q,y\rangle$. Set $R=C_P(y)$.  Since $y$ does not centralise $x$ we have
$x\not\in R$, so $y\not\in Q$, and in particular, $Q\neq R$. 
By the above arguments, $R$ is normal of index $p$
in $P$, we have $Q\neq R$, and hence $Q\cap R$ has index $p^2$ in $P$.
Since $|P|\geq p^5$ it follows that $Q\cap R$ has order  at least $p^3$, and thus 
we have $Z < Q\cap R$. That is, $Q\cap R$ is not a central subgroup of $P$, so
$C_P(Q\cap R) < P$. Assume now that $Q$ is abelian. Then
$Q\leq C_P(Q\cap R)<P$, so $Q=C_P(Q\cap R)\geq C_P(R)\geq Z(R)$.
But $y\in Z(R)$, hence $y\in Q$, a contradiction. Thus $Q$ is nonabelian.
\end{proof}

This Proposition  is obviously not true if $P$ is extraspecial of order $p^3$, since in 
that case, if $x\in P\setminus Z$, then  $C_P(x)$ is abelian of order $p^2$. 

\begin{Theorem}[{cf. \cite[Theorem 1.4, Theorem 11.6]{BKL2021}}] 
\label{extraspecial-Thm}
Let $P$ be an extraspecial finite $p$-group of order $p^{1+2r}$ for
some integer $r\geq 2$, and set $Z=Z(P)$. We have a vector space
decomposition
$$\HH^1(kP) = \CH \oplus \CN,$$
where $\CH$ is an abelian Lie subalgebra of dimension $2r$, and where
$\CN$ is a  Lie ideal of dimension $p^{2r}(2r-1)+2r(p-1)$, all of whose elements 
are represented by  derivations with image in $J(kP)^2$. In particular, $\CN$ is 
nilpotent, and $\HH^1(kP)$ is strongly solvable.  Moreover, $\CN$ is the kernel
of the canonical map $\HH^1(kP)\to$ $\HH^1(kP/Z)$. 
\end{Theorem}

\begin{proof}
The group $\bar P=P/Z$ is elementary abelian of order $p^{2r}$.
Since $Z$ is also the commutator subgroup of $P$, the canonical
map $P\to \bar P$ induces a canonical Lie algebra homomorphism $\psi$
$$\xymatrix{\HH^1(kP) \ar[rr] & & \HH^1(k\bar P)},$$
by Lemma \ref{Amodcommutators} (iv). Together with the canonical isomorphism
$H^1(P;kP)\cong$ $\HH^1(kP)$ and its analogue for $\bar P$, this yields a canonical map
$$\xymatrix{\H^1(P; kP) \ar[rr] & & \H^1(\bar P; k\bar P)},$$
where $P$ and $\bar P$ act by conjugation on $kP$ and $k\bar P$, respectively. 
With a slight abuse of notation, we also denote the above map as $\psi$.
Note that 
$$\H^1(P; k)= \Hom(P,k)\cong\Hom(\bar P,k)=\H^1(\bar P; k).$$
Let $\CH$ be the canonical image of $\H^1(P;k)$ in $\HH^1(kP)$, via the obvious 
identifications $k\cong k\cdot 1_P$. Then $\CH$ is mapped isomorphically onto the 
canonical image of $\H^1(\bar P; k)$ in  $\HH^1(k\bar P)$, via the identification 
$k\cong k\cdot 1_{\bar P}$. 
By Proposition \ref{H1Gk-image-Prop}, these spaces are abelian Lie algebras of 
dimension equal to the rank $2r$ of $\bar P$. 

Let $x\in P$, and let $C$ be the conjugacy class in $P$ containing $x$.
Then $P$ acts by conjugation on the transitive permutation $kP$-module
$$kC \cong \Ind^P_{C_P(x)}(k).$$

\medskip
Suppose first that  $x\in Z$ Then $C=\{x\}$. The vector space decomposition
$kZ =$ $k\cdot 1_Z \oplus I(kZ)$ yields
$$\H^1(P; kZ) \cong \H^1(P;k) \tenk kZ \cong \CH\oplus \H^1(P;k)\tenk I(kZ) .$$
Thus the image of $\H^1(P; kZ)$ regarded as a subspace of $\H^1(P; kP)$ under $\psi$ is
$$\psi (\H^1(P; kZ))=\psi(\CH\oplus \H^1(P;k)\tenk I(kZ))=\psi (\CH).$$
Since $Z(P)\leq \Phi(P)$, Remark \ref{H1Gk-Remark}
implies that $\H^1(P; kZ)$ is an abelian Lie subalgebra of $\HH^1(kP)$). 

\medskip
Suppose now that $x\in P\setminus Z$. Then $C = Zx$, so has order $p$, and
$C_G(x)$ has index $p$ in $P$ and is nonabelian by Proposition
\ref{nonabelian-centralisers}. Set $Q=C_P(x)$. As mentioned above, 
as a $kP$-module with respect
to the conjugation action of $P$ we have an isomorphism
$kC \cong \Ind^P_Q(k)$. The image of $kC$ in $k\bar P$ under the canonical
map $kP\to k\bar P$ is equal to $k\bar x$, where $\bar x$ is the image of
$x$ in $\bar P$. By a standard adjunction, we have a canonical isomorphism
$\H^1(Q; k x) \cong$ $\H^1(P, kC)$. By Lemma \ref{tr1-Lemma},
composing this map with the canonical
map $\H^1(P; kC)\to \H^1(P; k \bar x)$ yields the transfer
map $\Tr^P_Q : \H^1(Q; kx)\to \H^1(P; k \bar x)$, modulo identifying $kx$ and 
$k\bar x$ with
the trivial modules for $kQ$ and $kP$, respectively. 
Now  $Z$ is the commutator subgroup of $P$ as well as of $Q$ 
(by Proposition \ref{nonabelian-centralisers}), and hence,  setting $\bar Q=$ $Q/Z$, 
we have $\H^1(Q; kx)\cong$ $\H^1(\bar Q; k\bar x)$ of dimension $2r-1$ and
$\H^1(P; k \bar x)\cong$ $\H^1(\bar P; k\bar x)$. In other words, $\Tr^P_Q$
induces the transfer map
$\Tr^{\bar P}_{\bar Q} : \H^1(\bar Q; k\bar x) \to$ $\H^1(\bar P; \bar x)$.
It is well-known that this map is zero as $\bar P$ is elementary abelian and
$\bar Q$ a proper subgroup of $\bar P$. 
This shows that $\psi$ sends the subspace $H^1(P; kC)$ of $\HH^1(kP)$ to zero in 
$\HH^1(k\bar P)$.

This implies that $\Im(\psi)=\psi(\H^1(P;kZ))=\psi(\CH)\cong \CH$. 
Let $\CN$ be the kernel of $\psi$ on $\HH^1(kP)$, then we have that as vector space
$$\HH^1(kP)\cong \Im(\psi)\oplus \ker\psi\cong \CH\oplus \CN.$$
In particular, $\CH$ is abelian and by Lemma \ref{PPhiP-Lemma}, $\CN$ is a nilpotent 
Lie idea with elements represented by derivations with image in $J(kP)^2$. 
Thus, the derived Lie subalgebra of $\HH^1(kP)$ is contained $\CN$ and it is nilpotent. 
The result follows and the dimension of $\CN$ can be calculated using the centraliser 
decomposition.
\end{proof}

Note that the Lie ideal $\CN$ in this Theorem properly contains the
derived Lie subalgebra of $\HH^1(kP)$.
We include further  results on the centre of $kP$. 

\begin{Proposition} \label{extraspecial-centre}
Let $P$ be an extraspecial finite $p$-group.
Set $Z=Z(P)$.  The following hold.
\begin{itemize}
\item[{\rm (i)}]
$Z(kP) = kZ + (\sum_{w\in Z}w)kP = kZ \oplus  (\sum_{w\in Z} w) J(kP)$.
\item[{\rm (ii)}]
$J(Z(kP)) = I(kZ)  \oplus  (\sum_{w\in Z} w) J(kP)$.
\item[{\rm (iii)}]
We have
$J(Z(kP))\subseteq J(kP)^2$.
\item[{\rm (iv)}]
For any integer $m\geq 2$ we have
$J(Z(kP))^m = I(kZ)^m$.
\item[{\rm (v)}] We have
$J(Z(kP))^{p-1} = k\cdot (\sum_{w\in Z} w)$.
\item[{\rm (vi)}]
We have $(\sum_{w\in Z}w)J(kP)\subseteq J(kP)^{2p-1}$.
\end{itemize}
\end{Proposition}

\begin{proof}
If $g\in P\setminus Z$, then the conjugacy class of $g$ is the coset $Zg$, because
$Z$ is equal to the commutator subgroup of $P$ as $P$ is extraspecial.
Thus $Z(kP)$ is spanned by the elements in $Z$ and elements of the form
$(\sum_{w\in W} w)g$, where $g\in P\setminus Z$. In particular,
$(\sum_{w\in Z}w)kP\subseteq$ $Z(kP)$. Noting that
if $g\in Z$, then $(\sum_{w\in Z}w)g=$ $\sum_{w\in Z}w$, it follows that
$Z(kP)=$ $kZ + (\sum_{w\in P}w) kP$. Again, if $g\in P$, then
$(\sum_{w\in Z}w)g=$ $\sum_{w\in Z}w\ + \ (\sum_{w\in Z}w)(g-1)$.
This implies that in fact $Z(kP)=$ $kZ + (\sum_{w\in Z}w)J(kP)$. 
Since $\sum_{w\in Z}w$ annihilates $I(kZ)=$ $kZ\cap J(kP)$, it follows
that this sum is direct. This shows (i).
Clearly $\sum_{w\in Z}w$ is contained in $J(Z(kP))$,
so (ii)  follows from  (i). 
Alternatively one can show the directness of the sums in (i) and (ii) by comparing
dimensions. 
Lemma \ref{L-Jsquare}, together with (ii),  implies (iii).
The element $\sum_{w\in Z}w$ squares to zero and annihilates $I(kZ)$, whence (iv).
Finally, (v)  follows from the standard fact that since $Z$ is cyclic of
order $p$ we have $\soc(kZ)=$
$k\cdot (\sum_{w\in Z}w)=$ $I(kZ)^{p-1}$.
Since $I(kZ)\subseteq$ $J(kP)^2$, it follows further that 
$(\sum_{w\in Z}w)kP\subseteq$
$J(kP)^{2p-2}$, whence (vi).
\end{proof}

The following Corollary  reinterprets this Proposition: the centre $Z(kP)$, regarded as
a module over the Lie algebra $\HH^1(kP)$, has a $1$-dimensional submodule
$k\cdot (\sum_{w\in Z} w)$, and this
submodule  is annihilated by the derived Lie subalgebra of $\HH^1(kP)$.

\begin{Corollary}
Let $P$ be an extraspecial finite $p$-group.
Set $Z=Z(P)$. 
\begin{itemize}
\item[{\rm (i)}]
Every derivation on $kP$ preserves $J(Z(kP))$.
\item[{\rm (ii)}]
Every derivation  on $kP$ preserves the subspace $k\cdot (\sum_{w\in Z}w)$,
and any inner derivation annihilates this subspace. 
\item[{\rm (iii)}]
For $f$, $g$ two derivations on $kP$, the derivation $[f,g]$ annihilates $\sum_{w\in Z}w$.
\end{itemize}
\end{Corollary}

\begin{proof}
By Proposition \ref{extraspecial-centre} (iii) we have $J(kZ(P))\subseteq$ $J(kP)^2$.
It follows from Lemma \ref{JZAinJA2} that every derivation $f$ on $kP$
preserves $J(Z(kP))$, whence (i). But then every derivation on $kP$ preserves
also  any power of $J(Z(kP))$. Thus (ii) follows from
Proposition \ref{extraspecial-centre} (v),  and (iii) is an immediate consequence of (ii).
\end{proof}

\begin{Proposition} \label{extraspecial-centre-kernel}
Let $P$ be an extraspecial finite $p$-group.
Set $Z=Z(P)$.   Let $f$ be a derivation on $kP$. If $(\sum_{w\in Z}w)kP\subseteq$ 
$\ker(f)$, then $\Im(f)\subseteq$ $I(kZ)kP$. In particular, 
the kernel of the canonical map $\HH^1(kP)\to$ $\HH^1(Z(kP))$
is contained in the nilpotent subalgebra $D_2(kP)$ of classes represented by 
derivations with image contained in $J(kP)^2$.
\end{Proposition}

\begin{proof}
Note that the hypotheses on $f$ imply in particular that $f(\sum_{w\in Z}w)=0$.
Thus the Leibniz rule implies that 
$0=$ $f((\sum_{w\in Z})kP) = $ $(\sum_{w\in Z}w) f(kP)$.
It follows that  $\Im(f)$ is contained in the annihilator of $\sum_{w\in Z}w$ in $kP$.
Since $kP$ is a free $kZ$-module, this annihilator is $I(kZ)kP$.
Therefore we have  $\Im(f)\subseteq$ $I(kZ)kP$. This shows the first statement.
By Proposition \ref{extraspecial-centre} (i) we have $(\sum_{w\in Z}w)kP\subseteq$
$Z(kP)$. Thus if $f$ vanishes on $Z(kP)$, then the first statement implies that
$\Im(f)\subseteq$ $I(kZ)kP$. 
By Lemma \ref{L-Jsquare} we have $I(kZ)kP\subseteq$ $J(kP)^2$, whence
the second statement.
\end{proof}

For future reference, we mention the following elementary facts on 
extraspecial $p$-groups and their $p'$-automorphism groups. 

\begin{Proposition}
Let $P$ be an extraspecial finite $p$-group and $E$ a non-trivial $p'$-subgroup 
of $\Aut(P)$. Set $Z=Z(P)$, $T=C_P(E)$, and $F = [P,E]$.
 The following hold.
\begin{itemize}
\item[{\rm (i)}] We have $Z \leq F$.
\item[{\rm (ii)}] We have $F=[F,E]$.
\item[{\rm (iii)}] If $E$ acts trivially on $Z$, then $F$ and $T$
centralise each other.
\item[{\rm (iv)}]
If $E$ acts trivially on $Z$, then $F$ is non-abelian. 
\end{itemize}
\end{Proposition}

\begin{proof}
Statement (ii) is a special case of  \cite[Theorem 5.3.6]{Gor}.
 If $E$ acts non-trivially on $Z$, then statement (i)  is obvious, so for the 
 remainder  of the proof we assume that $E$ acts trivially on $Z$;
 that is, we assume that $Z\leq T$. 
 Since $E$ is non-trivial, so is $F$, and $T$ is a proper  subgroup of $P$. 
 
 For all $a\in F$, $b\in T$ and $e\in E$   we have
$[a,b] = [a,b]^e = [a^e, b].$
Thus, using that $[a^e,b]\in Z$, we have
$$1 = [a^e,b]\cdot [a,b]^{-1}=[a^e,b]\cdot [a^{-1},b] = [a^ea^{-1}, b].$$
It follows that $a^ea^{-1}\in C_P(T)$, for all $a\in F$ and all $e\in E$. 
Using (ii) we therefore have
 $F=$ $[F,E]\leq C_P(T)$; that is, $F$ and $T$ centralise each other.
This shows (iii).
 
 By \cite[Theorem 5.3.5]{Gor}  we have $P=TF$. By \cite[Theorem 5.2.3]{Gor} 
 the image in the abelian group $P/Z$  of this product becomes a direct product.
It follows that $T\cap F \leq Z$.  If $T\cap F$ were trivial, then $P$ would
be the direct product of the non-trivial proper subgroup $T$ and the group $F$, 
which is impossible. Thus we must have $Z\leq$ $F$,
which  completes the proof of (i).

If $F$ were abelian, then together with the fact proved earlier that $F$ and $T$
centralise each other, it would follow that $F$ is central in $P$, hence that $F\leq Z$, 
so $F=Z$. But then the equality $P=FT$ and the hypothesis $Z\leq T$ force $P=T$, 
contradicting the assumption that $T$ is a proper subgroup of $P$ since $E$ is
assumed to be  a non-trivial subgroup of $\Aut(P)$. 
This shows (iv), completing the proof.
\end{proof}

\bigskip\noindent
{\bf Acknowledgements.} 
The authors  acknowledge support from
EPSRC grant EP/X035328/1.

\medskip\noindent
{\bf Remark.}
No AI tools have been used in the writing of this paper.


\end{document}